\documentclass{lms}
\usepackage{latexsym,
psfrag,amsfonts,amsmath,amssymb,epsfig,hhline,rotating,
url,color}
\makeatletter
\@addtoreset{figure}{section}
\makeatother
\renewcommand{\thefigure}{\thesection.\arabic{figure}}

\makeatletter
\@addtoreset{table}{section}
\makeatother
\renewcommand{\thetable}{\thesection.\arabic{table}}

\newtheorem{theorem}{Theorem}[section]
\newtheorem{lemma}[theorem]{Lemma}
\newtheorem{cor}[theorem]{Corollary}
\renewcommand{\proofname}{Proof}

\newnumbered{defin}[theorem]{Definition}
\newnumbered{question}[theorem]{Question}
\newnumbered{remark}[theorem]{Remark}
\newunnumbered{nota}{Notation}
\newnumbered{example}[theorem]{Example}
\def \T{{\mathbb T}}
\def \H{{\mathbb H}}
\def \N{{\mathbb N}}
\def \R{{\mathbb R}}

\def \E{{\mathbb E}}
\def \Z{{\mathbb Z}}
\def \S{{\mathbb S}}
\def \G{{\mathcal G}}

\def \R{{\mathcal R}}
\def \C{{\mathcal C}}

\def \[{[ }
\def \]{] }

\def \t{\tilde}

\def \wt{\widetilde}

\def \vol{\text{\,vol\,}}
\def \dim{\text{\,dim\,}}

\def  \Star{\text{\,Star\,}}

\title[Coxeter groups, mutations and manifolds]%
{Coxeter groups, quiver mutations and geometric manifolds}
\classno{20F55, 13F60, 51F15}

\author{Anna Felikson and Pavel Tumarkin}

\begin{document}
\maketitle

\begin{abstract} 
We construct finite volume hyperbolic manifolds with large symmetry groups. The construction makes use of the presentations of finite Coxeter groups provided by Barot and Marsh and involves mutations of quivers and diagrams defined in the theory of cluster algebras. We generalize our construction by assigning to every quiver or diagram of finite or affine type a CW-complex with a proper action of a finite (or affine) Coxeter group. These CW-complexes undergo mutations agreeing with mutations of quivers and diagrams. We also generalize the construction to quivers and diagrams originating from unpunctured surfaces and orbifolds.
\end{abstract}

\tableofcontents

\section{Introduction}
In~\cite{BM} Barot and Marsh provided presentations of finite Weyl groups in terms of arbitrary quivers or diagrams of finite type. In brief, the construction works as follows: one needs to consider the underlying unoriented labeled graph of a quiver (or diagram) as a Coxeter diagram of a (usually infinite) Coxeter group, and then impose some additional relations on this group that can be read off from the quiver/diagram. It is proved in~\cite{BM} that the resulting group depends on the mutation class of a quiver/diagram only, see Section~\ref{group} for the details. 

The goal of this paper is to use these presentations to construct hyperbolic manifolds, in particular ones of small volume. The most common way to construct hyperbolic manifolds is by considering finite index torsion-free subgroups of cofinite hyperbolic reflection groups (see e.g.~\cite{CM,E,ERTs,RTs}). The approach used in the present paper is completely different: we start with a {\it symmetry group} of a manifold, more precisely, we are looking for a manifold whose symmetry group contains a given finite Weyl group. In~\cite{EM} Everitt and Maclachlan describe an algebraic framework for this approach, and in~\cite{KSl} Kolpakov and Slavich explicitly construct an arithmetic hyperbolic $4$-manifold with given finite isometry group. We use quite different techniques based on {\em mutations of quivers and diagrams} defined by Fomin and Zelevinsky in the context of cluster algebras. This technique allows us to obtain manifolds having large symmetry groups and, at the same time, relatively small volume, i.e. two properties of quite opposite nature.

Mutation classes of quivers and diagrams of finite type are indexed by finite Weyl groups. Our interpretation of the result of~\cite{BM} is that, starting from a quiver or diagram $\G$ of finite type, one can present the corresponding finite Weyl group $W$ as a quotient of some (usually infinite) Coxeter group $W_0(\G)$ by the normal closure $W_C(\G)$ of additional relations. If this Coxeter group $W_0(\G)$ is hyperbolic, then the Weyl group $W$ acts properly on the quotient of the hyperbolic space by the group $W_C(\G)$. {\it A~priori}, this quotient may have singular points. However, this is not the case.

\setcounter{section}{6}
\setcounter{theorem}{1}

\begin{theorem}[(Manifold Property)]
The group $W_C(\G)$ is torsion-free.
\end{theorem}     

As a corollary, if the Coxeter group $W_0(\G)$ is hyperbolic and cofinite, we obtain a hyperbolic manifold of finite volume, and the symmetry group of this manifold contains $W$. We list the manifolds that we found in this way in Tables~\ref{hyp} and~\ref{hyp-n}. In dimension $4$ we construct a manifold with Euler characteristic $2$, i.e. it has the second minimal volume amongst all $4$-dimensional manifolds. The symmetry group of this manifold contains the group $\text{Sym}_5\rtimes \Z_2$ as a subgroup. We note that, since this manifold has small volume and large symmetry group at the same time, it does not appear in~\cite{KSl}.

In~\cite{FT} we generalized the results of~\cite{BM} to quivers and diagrams of affine type and certain other mutation-finite quivers and diagrams. In particular, every affine Weyl group also admits presentations of a similar type with slightly more complicated groups $W_C(\G)$. Using this result, in Section~\ref{aff} we provide the following theorem.  

\setcounter{section}{8}
\setcounter{theorem}{0}

\begin{theorem}
The manifold property holds for quivers and diagrams of affine type.
\end{theorem}     
\setcounter{section}{1}
\setcounter{theorem}{0}

For quivers and diagrams originating from unpunctured surfaces or orbifolds~\cite{FST,FeSTu3}, as well as for exceptional mutation-finite quivers and diagrams~\cite{FeSTu1,FeSTu2}, the construction in~\cite{FT} also provides a group assigned to a mutation class. In contrast to the finite and affine cases, this group is not a Coxeter group, however, it also has presentations as a quotient of a Coxeter group $W_0$ (with more types of generators for $W_C$ required). 

Using these presentations, for every quiver or diagram $\G$ of all types discussed above, we construct a CW-complex (actually being a quotient of the Davis complex of a Coxeter group $W_0(\G)$) with a proper action of the group $W(\G)=W_0(\G)/W_C(\G)$, where the group $W(\G)$ depends on the mutation class of $\G$ only. We also define mutations of these complexes, the mutations agree with quiver (or diagram) mutations. All the manifolds discussed above are partial cases of these complexes.

\bigskip


The paper is organized as follows. Section~\ref{background} contains basic definitions and some essential facts on Coxeter groups and quiver mutations. In Section~\ref{group} we recall the construction from~\cite{BM} assigning to a quiver of finite type a presentation of the corresponding Weyl group. For simplicity, in Sections~\ref{group}--\ref{Davis} we restrict ourselves to simply-laced Weyl groups (and thus, we work with quivers rather than diagrams), we then consider the general case in Section~\ref{sec diagr}. In Section~\ref{various actions} we describe the construction highlighted above of various actions of a given finite Weyl group. In section~\ref{examples} we first treat in full details our construction applied to the group $A_3$; then we list other geometric manifolds with Weyl group actions resulted from the algorithm described in Section~\ref{various actions}. In particular, Table~\ref{hyp} contains the list of hyperbolic manifolds we obtain. In Section~\ref{Davis}, after defining the Davis complex of a Coxeter group, we prove the Manifold Property. We also define mutations of quotients of Davis complexes agreeing with quiver mutations. Section~\ref{sec diagr} is devoted to generalizations of the results to all finite Weyl groups. Finally, in Section~\ref{inf} we describe how the construction generalizes to affine Weyl groups and some other infinite groups related to quivers and diagrams of finite mutation type.     

\section{Quiver mutations and presentations of Coxeter groups}
\label{background}

\subsection{Coxeter groups}
We briefly remind some basic properties of Coxeter groups. For the details see~\cite{D}.

\subsubsection{Definitions}
A group $W$ is called a (finitely generated) {\it Coxeter group} if it has a presentation of the form
$$ 
W=\langle  s_1\dots s_n \ | \ s_i^2=(s_is_j)^{m_{ij}}=e  \rangle
$$ 
where $m_{ii}=1$ and $m_{ij}\in \N_{>1}\cup \infty$ for all $i\ne j$.
Here $m_{ij}=\infty$ means that there is no relation on $s_i s_j$.
A pair $(W, S)$ of a Coxeter group $W$ and its set of generators $S=\{s_1,\dots,s_n \}$ is called a {\it Coxeter system}. The cardinality $n$ of $S$ is a {\it rank} of the Coxeter system.

An element of $W$ is said to be a {\it reflection} if it is conjugated in $W$ to an element of $S$.



A Coxeter system $(W,S)$  may be depicted by its {\it Coxeter diagram}: 
\begin{itemize}
\item the vertices of the Coxeter diagram correspond to the generators $s_i$;
\item $i$-th vertex is connected to $j$-th by an edge labeled $m_{ij}$ with the following exceptions:
\begin{itemize}
\item if $m_{ij}=2$ then there is no edge connecting $i$ to $j$;
\item if $m_{ij}=3$ then the label 3 is usually omitted;
\item if $m_{ij}=\infty$ the edge is drawn bold.
\end{itemize}

\end{itemize}

Given a Coxeter system $(W,S)$, one can write a symmetric $n\times n$ matrix $M(W,S)=\{M_{ij}\}$ as follows:
$$
M_{ij}=
\begin{cases}
1 & \text{if $i=j$;} \\
-\cos\frac{\pi}{m_{ij}} & \text{if $i\ne j$ and $m_{ij}\ne \infty$;} \\
-1 & \text{if $m_{ij}=\infty$.} \\
\end{cases}
$$

An $n\times n$ matrix $a_{ij}$ is {\it decomposable} if 
there is a non-trivial partition of the index set as $\{1,\dots,n\}=I\cup J$, so that
$a_{ij}=a_{ji}=0$ whenever $i\in I$, $j\in J$. 
A matrix is  {\it indecomposable} if it is not decomposable.

\subsubsection{Actions by reflections}
\label{action}
For some Coxeter systems $(W,S)$ the group $W$ has a natural discrete action by reflections on a space of constant curvature
(here by an {\it action by reflections} we mean an action where the reflections of $W$ are represented by orthogonal reflections with respect to hyperplanes). 

More precisely, there are three cases of interest:
\begin{itemize}
\item  $M(W,S)$ is positive definite: then $W$ is a finite group and $W$ acts on a sphere $S^{n-1}$;
\item  $M(W,S)$ is indecomposable and has signature $(n-1,0,1)$: then $W$ acts cocompactly on Euclidean space $\E^{n-1}$;
\item  $M(W,S)$ is indecomposable and has signature $(n-1,1)$: then $W$ acts on an $(n-1)$-dimensional hyperbolic space $\H^{n-1}$.

\end{itemize}

See~\cite{V85} for the details.

A fundamental domain of such an action is a connected component of the complement to the set of all hyperplanes fixed pointwise by reflections of the group. This fundamental domain is a polytope (which may be of infinite volume) with facets indexed by generators $\{s_i\}_{i=1,\dots,n}$ and all dihedral angles equal to $\pi/m_{ij}$ (if $m_{ij}=\infty$ the corresponding facets do not intersect). Such polytopes are called {\it Coxeter polytopes}.

\begin{remark}
\label{root}
If a Coxeter group $W$ is a finite or affine Weyl group, then the outer normal vectors to the facets of the fundamental polytope can be identified with the simple roots of the corresponding root system.

\end{remark}

In the three cases listed above the group $W$ is called {\it geometric} (note that our class of geometric Coxeter groups is wider than in~\cite[Chapter 6]{D}). 
In fact, every Coxeter group acts properly cocompactly by reflections on some specially constructed space~\cite{D1}, we will discuss this action in Section~\ref{Davis}.

\subsection{Quiver mutations}

Let $Q$ be a quiver (i.e. a finite oriented multi-graph) containing no loops and no 2-cycles. If there are $q$ arrows pointing from $i$-th vertex to the $j$-th, then we draw one arrow with a weight $w_{ij}=q$.


%


\subsubsection{Definition of a quiver mutation}
For every vertex $k$ of the quiver $Q$ one can define an involutive operation  $\mu_k$ called {\it mutation of $Q$ in direction $k$}. This operation produces a new quiver  denoted by $\mu_k(Q)$ which can be obtained from $Q$ in the following way (see~\cite{FZ1}): 
\begin{itemize}
\item
orientations of all arrows incident to the vertex $k$ are reversed; 
\item
for every pair of vertices $(i,j)$ such that $Q$ contains arrows directed from $i$ to $k$ and from $k$ to $j$ the weight of the arrow joining $i$ and $j$ changes as described in Figure~\ref{quivermut}.
\end{itemize} 

\begin{figure}[!h]
\begin{center}
\psfrag{a}{\small $a$}
\psfrag{b}{\small $b$}
\psfrag{c}{\small $c$}
\psfrag{d}{\small $d$}
\psfrag{k}{\small $k$}
\psfrag{mu}{\small $\mu_k$}
\epsfig{file=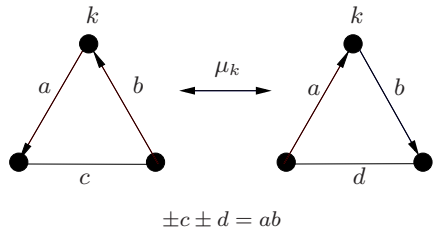,width=0.35\linewidth}\\
\medskip
$\pm{c}\pm{d}={ab}$
\caption{Mutations of quivers. The sign before ${c}$ (resp., ${d}$) is positive if the three vertices form an oriented cycle, and negative otherwise. Either $c$ or $d$ may vanish. If $ab$ is equal to zero then neither the value of $c$ nor orientation of the corresponding arrow changes.}
\label{quivermut}

\end{center}
\end{figure}

Given a quiver $Q$, its {\it mutation class} is a set of all quivers obtained from $Q$ by all sequences of iterated mutations. All quivers from one mutation class are called {\it mutation-equivalent}.

Quivers without loops and 2-cycles are in one-to-one correspondence with integer skew-symmetric matrices $B=\{b_{ij}\}$, where $b_{ij}$ is the number of arrows from $i$-th vertex to $j$-th one. In terms of the matrix $B$ the mutation $\mu_k$ can be written as $\mu_k(B)=B'$, where
$$b'_{ij}=\left\{
           \begin{array}{ll}
             -b_{ij}, & \hbox{ if } i=k \hbox{ or } j=k; \\
             b_{ij}+\frac{|b_{ik}|b_{kj}+b_{ik}|b_{kj}|}{2}, & \hbox{ otherwise.}\\
           \end{array}
         \right.
$$
This transformation is called a {\it matrix mutation}.

\begin{remark}
The procedure of matrix mutation is well-defined for more general case of {\it skew-symmetrizable} integer matrices. Matrix mutations for non-skew-symmetric matrices correspond to {\it diagram mutations} (see Section~\ref{sec diagr} for the definitions and generalizations of the results to the case of diagram mutations).

\end{remark}

\subsubsection{Finite type}
A quiver is of {\it finite type} if it is mutation-equivalent to an orientation of a simply-laced Dynkin diagram.
All orientations of a given Dynkin diagram are mutation-equivalent, so any quiver of finite type is of one of the following mutation types: $A_n$, $D_n$, $E_6$, $E_7$, $E_8$.



\subsubsection{Finite mutation  type}

A quiver $Q$ is of  {\it finite mutation type} (or {\it mutation-finite}) if there are finitely many quivers mutation-equivalent to $Q$.
It is shown in~\cite{FeSTu1}  that each quiver of finite mutation type is either of order $2$, or a quiver arising from a triangulated surface (see~\cite{FST} for details), or mutation-equivalent to one of the 11 exceptional quivers.

In this paper we deal with quivers of finite mutation type only.

\section{Presentations of Coxeter groups arising from quiver mutations}
\label{group}
In~\cite{BM} Barot and Marsh show that for each quiver of finite type there is a way to explicitly construct the corresponding Weyl group.

\subsection{Construction of the group by a quiver}
Let $Q$ be an orientation of a simply-laced  Dynkin diagram with $n$ nodes, let $W$ be the corresponding finite  Coxeter group, and $Q_1$ be any quiver mutation-equivalent to $Q$. Denote by $W(Q_1)$ the group generated by $n$ generators $s_i$ with the following relations:

\begin{itemize}
\item[(R1)] $s_i^2=e$ for all $i=1,\dots,n$;

\item[(R2)] $(s_is_j)^{m_{ij}}=e$ for all $i,j$, 
where
$$
m_{ij}=
\begin{cases}
2 & \text{if $i$ and $j$ are not connected;} \\
3 & \text{if $i$ and $j$ are connected by an arrow.} \\
\end{cases}
$$

\item[(R3)] {\bf (cycle relations)} for every chordless oriented cycle $\C$ given by 
$$i_0\stackrel{}\to i_1\stackrel{}\to\cdots\stackrel{}\to i_{d-1}\stackrel{}\to i_0$$
we take the relation
$$
(s_{i_0}\ s_{i_{1}}\dots s_{i_{d-2}}s_{i_{d-1}}s_{i_{d-2}}\dots s_{i_{1}})^{2}=e,
$$

\end{itemize}
Note that for a cycle of length $d$ we have a choice of $d$ relations of type (R3), we take any one of them.

It is shown in~\cite[Theorem A]{BM} that the group $W(Q_1)$  does not depend on the choice of a quiver in the mutation class of $Q$. In particular, 
it is isomorphic to the initial Coxeter group $W$. 

\begin{remark}
The results of~\cite{BM} are proved in more general settings, we come to the general case in Section~\ref{sec diagr}. 

\end{remark}

\subsection{Generators of $W(Q_1)$}
\label{pres}
Generators of $W(Q_1)$ satisfying relations (R1)--(R3) can be expressed in terms of standard generators of $W$. It is shown in~\cite{BM} that these generators of $W(Q_1)$ can be found inductively in the following way:

\begin{itemize}
\item[1.]
for $Q_1=Q$ the generators $\{s_i\}$ are standard Coxeter generators (indeed, Dynkin diagrams contain no cycles, so, we will get no cycle relations and obtain standard presentation of the Coxeter group).

\item[2.] Let $\{s_i\}$ be the generators of $W(Q_1)$, and let $Q_2=\mu_k(Q_1)$.
Then the generators $\{t_i\}$ of $W(Q_2)$ satisfying relations (R1)--(R3) are
$$
t_{i}=
\begin{cases}
s_k s_i s_k & \text{if there is an arrow from $i$ to $k$ in $Q_1$;} \\
s_i & \text{otherwise.} \\
\end{cases}
$$

\end{itemize}



\section{Construction of various actions of $W$ }
\label{various actions}
As before, let $Q$ be an orientation of a Dynkin diagram, and let $W$ be the corresponding Weyl group. Choose any quiver $Q_1$ mutation-equivalent to $Q$, and denote by $W_0(Q_1)$ the group defined by relations of type (R1) and (R2) only (note that $W_0(Q_1)$ is a Coxeter group). 

Let $C_1,\dots,C_p$ be a collection of all cycle relations for $Q_1$ (we take one relation for each chordless oriented cycle).
Denote by $W_C(Q_1)$ the normal closure of the union of the elements $C_1,\dots,C_p$ in $W_0(Q_1)$.
By the definition of $W(Q_1)$ we have $W(Q_1)=W_0(Q_1)/W_C(Q_1)$. Theorem~A from~\cite{BM} says that $W(Q_1)$ is isomorphic to $W=W(Q)$ which, in particular, means that $W_C(Q_1)$ has finite index in $W_0(Q_1)$. 

The group  $W_0(Q_1)$ depends on the quiver $Q_1$ chosen in the mutation class of $Q$. Moreover, depending on the choice of $Q_1$, the group  $W_0(Q_1)$ may be finite or infinite. As a consequence, it may  act naturally by reflections on different spaces (sphere, Euclidean space, hyperbolic space or, in general, the Davis complex of the group $W_0(Q_1)$, see Section~\ref{Davis}). Since $W= W(Q_1)=W_0(Q_1)/W_C(Q_1)$, this results in actions of the initial Coxeter group $W$ on different spaces. 

In Section~\ref{Davis} we show that the group $W_C(Q_1)$ is torsion-free (we call this {\it Manifold Property}). In particular, this implies that if the Coxeter group $W_0(Q_1)$ acts properly on a space ${\mathbb X}=\E^n$ or $\H^n$ by isometries with a fundamental domain of finite volume, then the quotient ${\mathbb X}/W_C(Q_1)$ is a finite volume manifold with a symmetry group containing $W$. In the next section we show examples of the construction of hyperbolic manifolds in this way, as well as of actions of Weyl groups of type $D$ on flat tori.

\section{Examples of manifolds constructed via quiver mutations}
\label{examples} 
\subsection{Action of $A_3$ on the flat $2$-torus}
\label{A3}
We now present the first (and the easiest) example showing how different actions of the same Coxeter group may by produced via mutations of quivers. 

\subsubsection{Action on the sphere}
\label{sph}
We start with the quiver $Q$  of type $A_3$ (see Fig.~\ref{stereo}(a),(b)). The corresponding Dynkin diagram of type $A_3$ determines a finite Coxeter group $W$ acting on the $2$-dimensional sphere by reflections:
$$W=\langle s_1,s_2,s_3 \,|\, s_i^2=(s_1s_2)^3=(s_2s_3)^3=(s_1s_3)^2=e \rangle $$ 
The fundamental domain of this action is a spherical triangle with angles $(\frac{\pi}{3},\frac{\pi}{3},\frac{\pi}{2})$, see Fig.~\ref{stereo}(c) for the stereographic projection image of the action and the fundamental domain.

\begin{figure}[!h]
\begin{center}
\psfrag{a}{\small (a)}
\psfrag{b}{\small (b)}
\psfrag{c}{\small (c)}
\psfrag{1}{\scriptsize $1$}
\psfrag{2}{\scriptsize $2$}
\psfrag{3}{\scriptsize $3$}
\epsfig{file=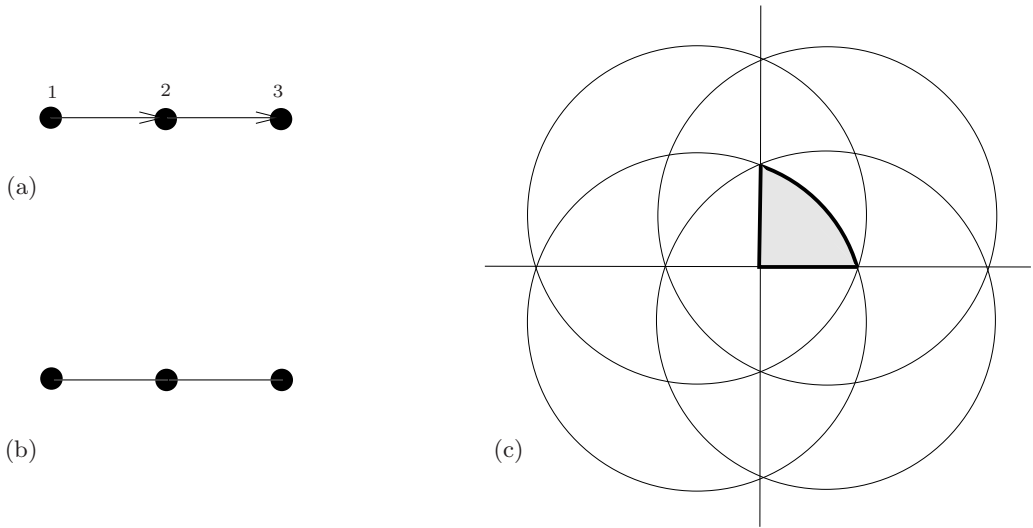,width=0.85\linewidth}
\caption{(a) A quiver of type $A_3$; (b) Dynkin diagram of type $A_3$; (c) action of $A_3$ on the sphere (in stereographic projection): a sphere  is tiled by 24 copies of (spherical) triangles with angles 
$(\frac{\pi}{2},\frac{\pi}{3},\frac{\pi}{3})$.
}
\label{stereo}
\end{center}
\end{figure}

\subsubsection{Action on the flat $2$-torus}
\label{torus}
Applying the mutation $\mu_2$ to the quiver $Q$, we obtain a quiver $Q_1$ (see Fig.~\ref{plane}(a)). The group $$W_0(Q_1)=\langle t_1,t_2,t_3  \ | \ t_i^2=(t_it_j)^3=e\rangle$$  
is the group generated by reflections in the sides of a regular triangle in the Euclidean plane (the corresponding Dynkin diagram is of type $\wt A_2$), see Fig.~\ref{plane}(b). Note that the plane is tessellated by infinitely many copies of the fundamental domain, each copy $F_g$ labeled by an element $g\in W_0(Q_1)$. To see the action of the initial group $A_3=W= W(Q_1)$ we need to take a quotient by the cycle relation
$$
(t_1\ t_2t_3t_2)^2=e,
$$
which means that in the tessellated plane we identify $F_e$ with the fundamental triangle labeled by the $(t_1\ t_2t_3t_2)^2$.

\begin{figure}[!h]
\begin{center}
\psfrag{a}{\small (a)}
\psfrag{b}{\small (b)}
\psfrag{c}{\small (c)}
\psfrag{Q}{\small $Q$}
\psfrag{Q1}{\small $Q_1$}
\psfrag{m2}{\scriptsize $\mu_2$}
\psfrag{1}{\scriptsize $1$}
\psfrag{2}{\scriptsize $2$}
\psfrag{3}{\scriptsize $3$}
\psfrag{l1}{\scriptsize $l_1$}
\psfrag{l2}{\scriptsize $l_2$}
\psfrag{l3}{\scriptsize $l_3$}
\psfrag{l232}{\scriptsize $l_{232}$}
\psfrag{Fe}{\scriptsize $F_e$}
\psfrag{F1232}{\scriptsize $F_{(t_1\ t_2t_3t_2)}$}
\epsfig{file=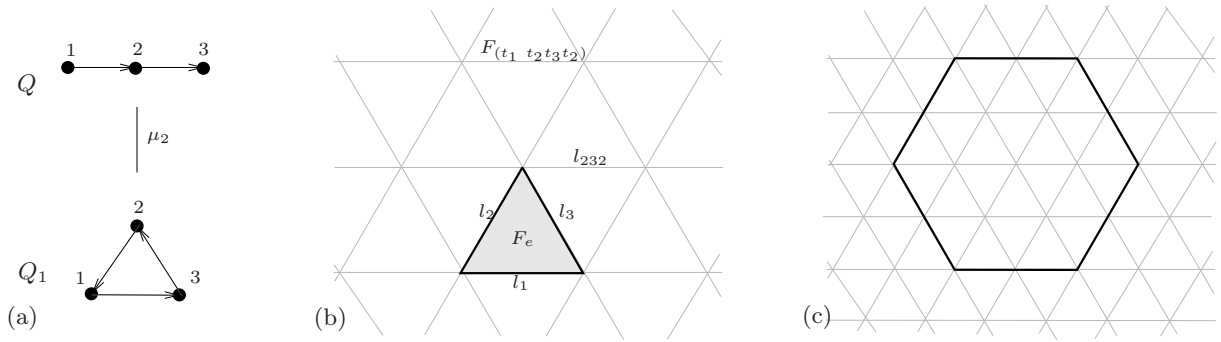,width=1.00\linewidth}
\caption{(a) mutation $\mu_2$; (b) action of $W_0(Q_1)=\widetilde A_2$ on the plane; (c) action of $W=W(Q_1)=A_3$ on the torus: $24$ regular triangles tile the hexagon whose opposite sides are identified.  }
\label{plane}
\end{center}
\end{figure}

Now, let us see what the transformation $T_1=(t_1\ t_2t_3t_2)^2$ means geometrically. 
The element $t_1$ acts on the plane as the reflection in the side $l_1$ of $F_e$,
$t_2t_3t_2$ is the reflection with respect to the line $l_{232}$ parallel to $l_1$ and passing through the vertex of $F_e$ not contained in $l_1$. Denote by $d$ the distance between $l_1$ and $l_{232}$. 
Then the transformation $(t_1\ t_2t_3t_2)$ is a translation of the plane by the distance $2d$ in the direction orthogonal to $l_1$. Hence, the element $T_1=(t_1\ t_2t_3t_2)^2$ is a translation  by $4d$ in the direction orthogonal to $l_1$.

Clearly, $e=t_2(t_1\ t_2t_3t_2)^2t_2=(t_2t_1t_2 \ t_3)^2$ in $W$, and 
$e=t_3 (t_1\ t_2t_3t_2)^2t_3= t_3(t_1\ t_3t_2t_3)^2t_3= (t_3t_1t_3\ t_2)^2$ in $W$.
This implies that we also need to take a quotient of the plane by the actions of the elements 
$T_3=(t_2t_1t_2 \ t_3)^2$ and  $T_2=(t_3t_1t_3\ t_2)^2$, i.e. by the translations by $4d$ in directions orthogonal to the lines $l_2$ and $l_3$.


The quotient of the plane by the group $W_C(Q_1)$ generated by three translations $T_1,T_2$ and $T_3$  is a flat torus, see Fig.~\ref{plane}(c). 
Thus, we obtain an action of the quotient group $A_3=W= W(Q_1)$ on the torus: the fundamental domain of this action is a regular triangle, $24$ copies of this triangle tile the torus (as $24$ copies tile the hexagon).



\begin{remark}
\label{non-reflection}
The action described above is not an action by (topological) reflections: the fixed sets of the generators $t_i$ are connected, so they do not divide the torus into two connected components (cf.~\cite[Chapter 10]{D}).

\end{remark}

\subsubsection{From the sphere to the torus:  ramified covering}
\label{sph to tor}
Now, we will show how the action of $A_3$ on the torus is related to the initial action on the sphere.

Recall from Section~\ref{pres} that the presentation for $W(Q_1)$ is obtained from the presentation for $W(Q)$ by the following change of generators
$$ (s_1,s_2,s_3)\to (t_1,t_2,t_3) \text{ \quad where \quad } t_1 =s_2s_1s_2, \ t_2=s_2, \ t_3=s_3.$$
In the above sections we considered the action of the $s$-generators on the sphere and of the $t$-generators on the plane and the torus.
However, the $t$-generators may be also considered directly on the sphere.

Namely, let $s_1, s_2,s_3$ be the reflections with respect to the corresponding sides $l_1,l_2,l_3$ of the fundamental spherical triangle $\Delta_s$, see Fig.~\ref{cover}(a).
Then the elements $t_2$ and $t_3$ are reflections with respect to $l_2$ and $l_3$ respectively.
The element $t_1$ is also a reflection as it is conjugate to the reflection $s_1$; more precisely, $t_1$ is a reflection with respect to the (spherical) line $l_{212}$ (see Fig.~\ref{cover}(b)).

\begin{figure}[!h]
\begin{center}
\psfrag{Dt}{\scriptsize $\Delta_t$}
\psfrag{Ds}{\scriptsize $\Delta_s$}
\psfrag{l1}{\scriptsize $l_1$}
\psfrag{l2}{\scriptsize $l_2$}
\psfrag{l3}{\scriptsize $l_3$}
\psfrag{l212}{\scriptsize $l_{212}$}
\psfrag{a}{\small (a)}
\psfrag{b}{\small (b)}
\epsfig{file=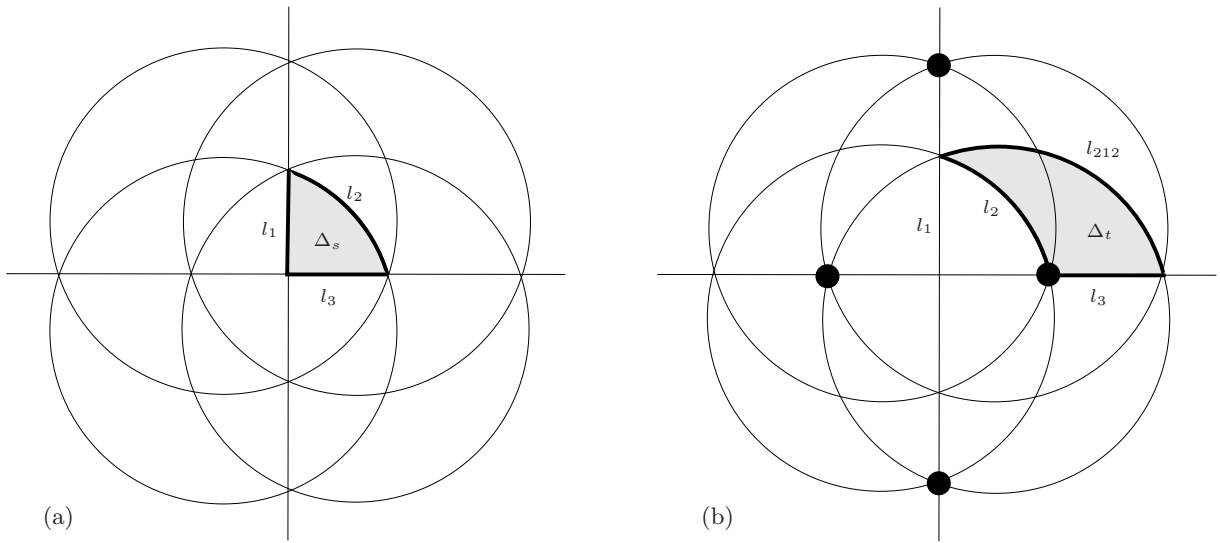,width=1.00\linewidth}
\caption{Construction of the ramified covering}
\label{cover}
\end{center}
\end{figure}

Consider the triangle  $\Delta_t$ bounded by $l_{212}, l_{2}$, and $l_{3}$ (there are 8 possibilities to choose such a triangle,
see Remark~\ref{rem choice} for the justification of our choice).

In contrast to the triangle $\Delta_s$ bounded by $l_1,l_2,l_3$, the triangle $\Delta_t$ is not a fundamental triangle for the action: $\Delta_t$ has angles $(\frac{\pi}{3},\frac{\pi}{3},\frac{2\pi}{3})$, and the group generated by $t_2$ and $t_3$ has order 6, but the whole neighborhood of the vertex $p\in l_2\cap l_3$ is tiled by $3$ copies of $\Delta_t$ only. 

To resolve this inconsistency we consider a $2$-sheet covering of the sphere branched in $p$ and in all $W$-images of $p$ (the four bold points in
Fig.~\ref{cover}(b)). Clearly, the covering space is a torus (moreover, one can easily see that the covering is given by Weierstrass $\wp$-function). It is also not difficult to see that the action of $A_3$ on this covering torus is exactly the same as its action on the torus in Section~\ref{torus}: the covering torus consists of $24$ triangles, and every vertex is incident to exactly six triangles.  



\begin{remark}[On the choice of the triangle $\Delta_t$]
\label{rem choice}
Consider the outer normals $v_1,v_2,v_3$ to the sides of $\Delta_s$  (cf. Remark~\ref{root}). 
To get the outer normals $u_1,u_2,u_3$  to the sides of $\Delta_t$ we use the following rule:
$$
u_i=
\begin{cases}
-v_i & \text{if $i=k$;} \\
s_k(v_i) & \text{if there is an arrow from $i$ to $k$ in $Q$;} \\
v_i & \text{otherwise.} \\ 
\end{cases}
$$
Here $s_k(v_i)$ is the vector obtained from $v_i$ by the reflection $s_k$.

\end{remark}

\subsubsection{Back from the torus to the sphere}
\label{back-s}
Now we will obtain the action of $A_3$ on the sphere from its action on the torus.

Let $\Delta_t$ be the fundamental triangle for the action of $A_3$ on the flat torus (see Fig.~\ref{back}(a)). Let $\pi$ be the canonical projection from Euclidean plane to the torus (corresponding to the action of $W_C(Q_1)$ on the plane), denote by $\wt\Delta_t$ any preimage of $\Delta_t$ under $\pi$.
The generators $t_1,t_2,t_3$ can be thought as reflections with respect to the sides of  $\wt\Delta_t$.
The generators $s_i$ can be obtained from the generators $t_i$ as follows:
$$
s_1=t_2t_1t_2, \ s_2=t_2, \ s_3=t_3. 
$$
These are reflections with respect to the sides of the dashed domain $\wt\Delta_s$ (see Fig.~\ref{back}(b)),
where $\wt\Delta_s$ is an unbounded strip on the Euclidean plane.

\begin{figure}[!h]
\begin{center}
\psfrag{Dt}{\scriptsize $\wt\Delta_t$}
\psfrag{Ds}{\scriptsize $\wt\Delta_s$}
\psfrag{l1}{\scriptsize $l_1$}
\psfrag{l2}{\scriptsize $l_2$}
\psfrag{l3}{\scriptsize $l_3$}
\psfrag{l212}{\scriptsize $l_{212}$}
\psfrag{a}{\small (a)}
\psfrag{b}{\small (b)}
\epsfig{file=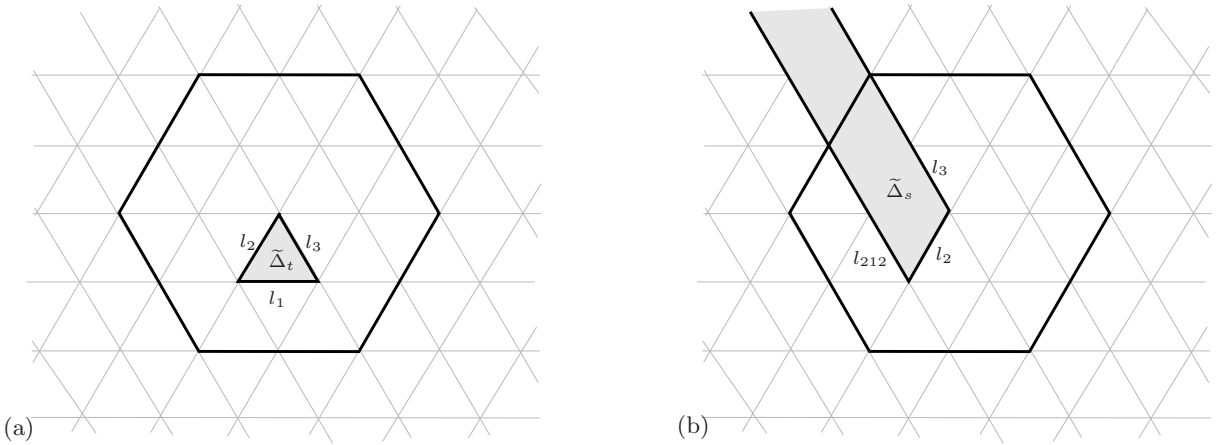,width=1.00\linewidth}
\caption{Back from the flat torus to the sphere}
\label{back}
\end{center}
\end{figure}

We now glue a manifold from $24$ copies of the strip (according to the group action, as we glued $24$ triangles above), denote by $X$ the space obtained via this gluing. Note that due to the cycle relations
four strips attached one to another along the ``long'' (unbounded) sides will result in an (unbounded) cylinder. The manifold $X$ consists of exactly $6=24:4$ cylinders.
Compactifying $X$ by adding a limit point at the end of each of the $6$ cylinders, we get exactly the tiling shown in Fig.~\ref{stereo}(c). 

\subsection{Actions of $D_n$ on the flat $(n-1)$-torus }
The next series of examples is provided by the finite Coxeter group $W=D_n$, $n> 3$ 
(see Fig.~\ref{dn}(a) for the Dynkin diagram), which naturally acts on $(n-1)$-dimensional sphere (for $n=3$ the Dynkin diagram and Weyl group $D_3$ coincides with $A_3$, so the example from the previous section can be considered as a partial case).
It is easy to check that the quiver $Q$ of type $D_n$ (see  Fig.~\ref{dn}(b)) is mutation-equivalent to the oriented cycle $Q_1$ shown in Fig.~\ref{dn}(c)
(in contrast to $A_3$-example, one needs more than one mutation for this).

\begin{figure}[!h]
\begin{center}
\psfrag{1}{\scriptsize $1$}
\psfrag{2}{\scriptsize $2$}
\psfrag{3}{\scriptsize $3$}
\psfrag{4}{\scriptsize $4$}
\psfrag{n1}{\scriptsize $n-1$}
\psfrag{n}{\scriptsize $n$}
\psfrag{a}{\small (a)}
\psfrag{b}{\small (b)}
\psfrag{c}{\small (c)}
\epsfig{file=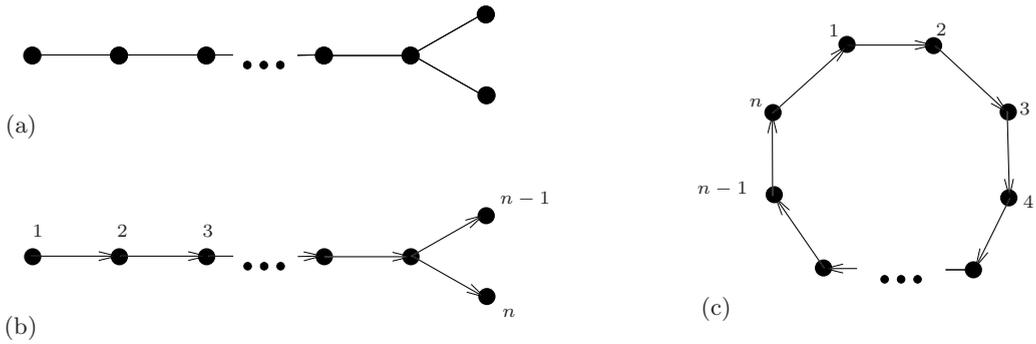,width=0.85\linewidth}
\caption{Type $D_n$: (a) Dynkin diagram; (b) acyclic quiver of type $D_n$; (c) oriented cycle.}
\label{dn}
\end{center}
\end{figure}

Note that the group $W_0(Q_1)$ (i.e. Coxeter group without cycle relations) is the affine Coxeter group $\widetilde A_{n-1}$,
thus, a group acting on the Euclidean space $\E^{n-1}$. A fundamental domain of this action is a Coxeter simplex $\Delta$ with the Coxeter diagram $\widetilde A_{n-1}$. The generators for this action are the reflections $t_1,\dots,t_n$ with respect to the facets $\Pi_1\dots,\Pi_n$ of this fundamental simplex.

Our next aim is to take a quotient by the cycle relation 
$$(t_1\ t_2t_3\dots t_{n-1}t_nt_{n-1}\dots t_3t_2)^2=e. $$
The element $(t_2t_3\dots t_{n-1}t_nt_{n-1}\dots t_3t_2) $ is a conjugate of $t_n$, so it is a reflection with  
respect to the hyperplane $\Pi$ parallel to $\Pi_1$ and passing though the vertex of $\Delta$ opposite to $\Pi_1$.
If $d$ is the distance between the hyperplanes $ \Pi$ and  $\Pi_1$, 
then the element $(t_1\ t_2t_3\dots t_{n-1}t_nt_{n-1}\dots t_3t_2)$ is a translation by $2d$ in the direction orthogonal to $\Pi_1$,
and its square  $T_1=(t_1\ t_2t_3\dots t_{n-1}t_nt_{n-1}\dots t_3t_2)^2$ is a translation by $4d$. 
So, we need to take the quotient by the action of this translation. 

The element $T_1$ has $n-1$ conjugates of type $T_k=(t_k\ t_{k+1}\dots t_{n}t_1\dots t_{k-2}t_{k-1}t_{k-2}\dots t_1t_n\dots t_{k+1})^2$, ($k=2,\dots,n$).
They all are translations by $4d$ in the direction orthogonal to $\Pi_k$.
The group generated by $T_1,\dots,T_n$ 
is a free abelian group of rank $n-1$. Hence, taking a quotient of $E^{n-1}$ by the translations $T_1,\dots,T_n$ we get an $(n-1)$-dimensional torus $\T^{n-1}$. The action of  $W_0(Q_1)$ on $\E^{n-1}$  turns into a faithful action of $D_n= W(Q_1)$ on the torus $\T^{n-1}$.



\begin{remark}
The Coxeter diagram $\widetilde A_n$ is the unique affine Dynkin diagram containing a cycle. This implies that our construction gives no further examples of actions on non-trivial quotients of Euclidean space.

\end{remark} 

\subsection{Actions on hyperbolic manifolds}
\label{act-hyp}
We now consider the case when the group  $W_0(Q_1)$ acts on a hyperbolic space, and hence, the group $W=W(Q_1)$ 
acts on its quotient $X=\H^n/W_C(Q_1)$ (recall that $W_C(Q_1)$ is the normal closure of all cycle relations). As it is shown below (Theorem~\ref{manifold}, Corollary~\ref{man}), in this case $X$ is always a manifold. In this section we list all known actions on hyperbolic manifolds obtained in this way (Table~\ref{hyp}).

\begin{nota}
 In Table~\ref{hyp} we omit orientations of arrows of quivers $Q$ and $Q_1$. Here arrows of $Q$ may be oriented in arbitrary way, and the only requirement on the orientations of arrows of $Q_1$ is that all chordless cycles are oriented.

The underlying graphs of $Q$ and $Q_1$ are the Coxeter diagrams of fundamental polytopes of the action of $W$ and $W_0(Q_1)$ on the sphere and the hyperbolic space respectively. All the volumes except the last two are computed in~\cite{vol}, the Euler characteristic (and thus the volume) of the fundamental polytope of the action of $D_8$ was computed by the program \url{CoxIter} written by Rafael Guglielmetti. The volume of the fundamental polytope of the action of $A_7$ is not known to the authors. The Euler characteristic is computed for all even-dimensional manifolds. The precise value of the volume of $X$ for four-dimensional $X$ is equal to $4\pi^2\chi(X)/3$, and for six-dimensional $X$ equals $-8\pi^3\chi(X)/15$.


\end{nota}
\setlength{\tabcolsep}{0.35cm}
\begin{table}[!h]
\vbox to 0.96\textheight{\vss
\hbox to \textwidth{\hss
\begin{turn}{90}
\begin{minipage}{\textheight}

\caption{Actions on hyperbolic manifolds.   }
\label{hyp}
\begin{tabular}{|c|c|c|c|c|c|c|c|c}
\hline
&&&&&&&\\
$W$    & $Q$ & $Q_1$ & $|W|$ & $\dim{X}$ &
\begin{tabular}{c}$\vol{X}$ \\ {\small approx.} \end{tabular} & \begin{tabular}{c}number \\ of cusps \end{tabular} & $\chi(X)$\\
&&&&&&&\\
\hline
$A_4$  & \parbox[c]{0.09\linewidth}{\epsfig{file=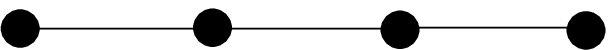,width=\linewidth}}    &  \parbox[c]{0.065\linewidth}{\epsfig{file=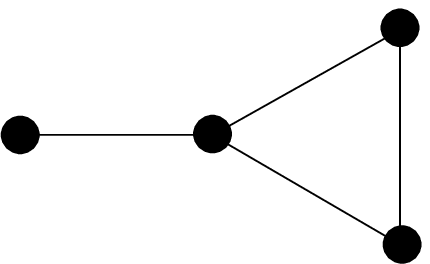,width=\linewidth}}\parbox[c]{0.091\linewidth}{\vphantom{\epsfig{file=./pic/a4_.eps,width=\linewidth}}}      & $5!$& 3    & {\small $|W|\cdot 0.084578$} &  5   &     \\ 
\hline
$D_4$  & \parbox[c]{0.065\linewidth}{\epsfig{file=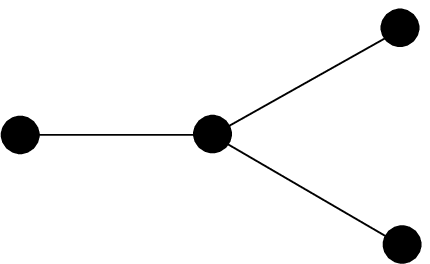,width=\linewidth}}     &  \parbox[c]{0.04\linewidth}{\epsfig{file=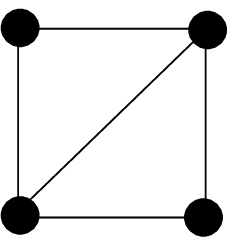,width=\linewidth}}\parbox[c]{0.06\linewidth}{\vphantom{\epsfig{file=./pic/d4_.eps,width=\linewidth}}}      & $2^3\cdot 4!$ & 3 &  {\small $|W|\cdot 0.422892$  }   & 16 & \\
\hline
$D_5$ &  \parbox[c]{0.09\linewidth}{\epsfig{file=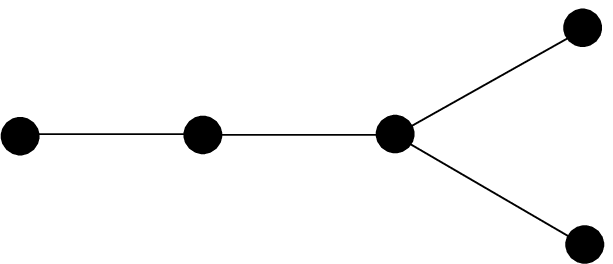,width=\linewidth}}    &  \parbox[c]{0.09\linewidth}{\epsfig{file=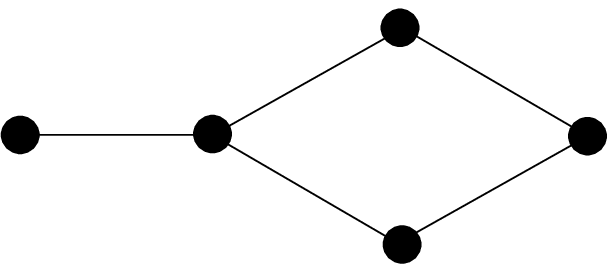,width=\linewidth}}\parbox[c]{0.13\linewidth}{\vphantom{\epsfig{file=./pic/d5_.eps,width=\linewidth}}}      & $2^4\cdot 5!$ & 4 &  
{\small $|W|\cdot 0.013707  $}  & 10 & 2 \\
\hline
$E_6$ & \parbox[c]{0.11\linewidth}{\epsfig{file=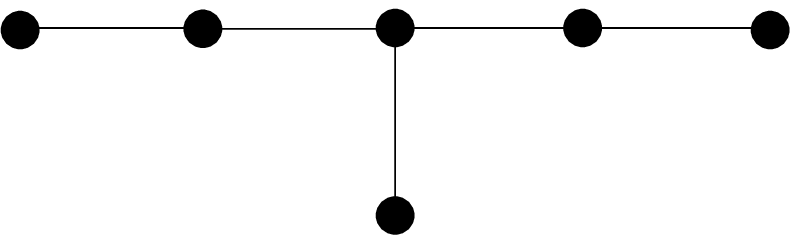,width=\linewidth}}    &   \parbox[c]{0.09\linewidth}{\epsfig{file=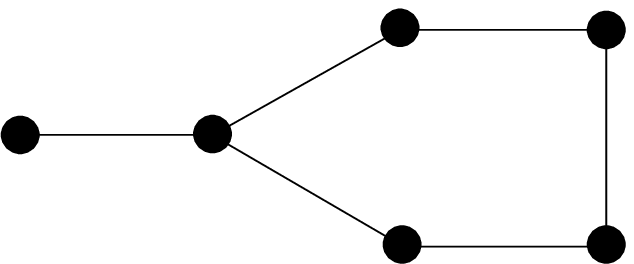,width=\linewidth}}\parbox[c]{0.13\linewidth}{\vphantom{\epsfig{file=./pic/e6_.eps,width=\linewidth}}}     & $2^7\cdot 3^4\cdot 5$  &5 &  {\small $|W|\cdot 0.002074 $} & 27 & \\
\hline 
$E_7$ &  \parbox[c]{0.13\linewidth}{\epsfig{file=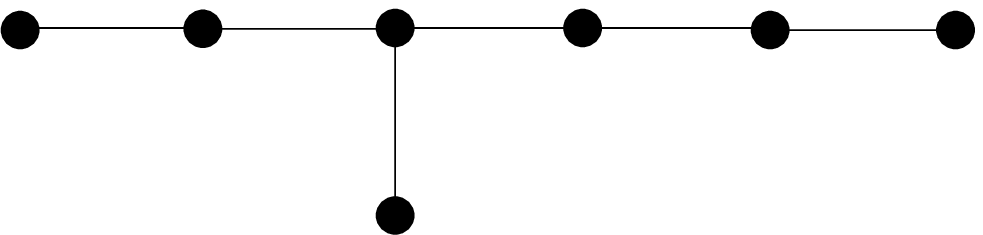,width=\linewidth}}   &   \parbox[c]{0.12\linewidth}{\epsfig{file=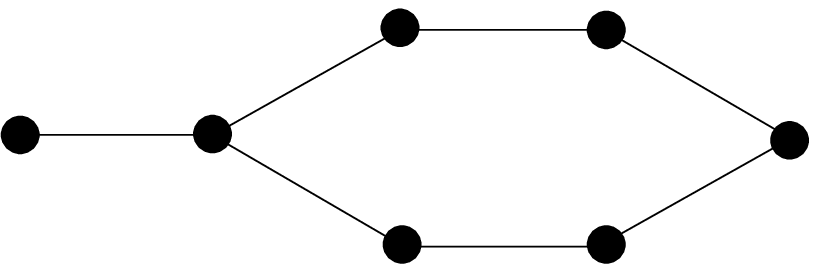,width=\linewidth}} \parbox[c]{0.17\linewidth}{\vphantom{\epsfig{file=./pic/e7_.eps,width=\linewidth}}}    & $2^{10}\cdot 3^4\cdot 5\cdot 7$  & 6 & {\small  $|W|\cdot 2.962092\times 10^{-4}$ }& 126 & -52 \\
\hline
$E_8$ &  \parbox[c]{0.16\linewidth}{\epsfig{file=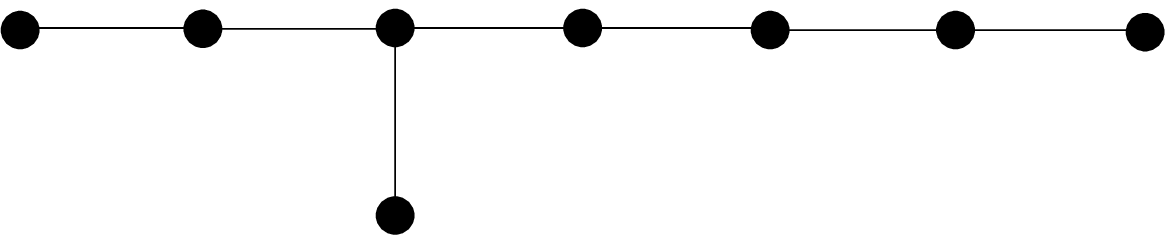,width=\linewidth}}   &   \parbox[c]{0.122\linewidth}{\epsfig{file=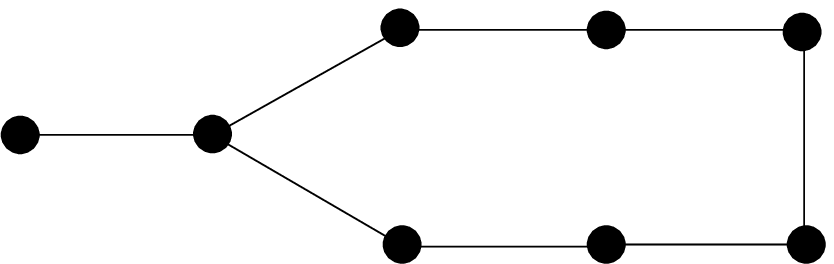,width=\linewidth}} \parbox[c]{0.17\linewidth}{\vphantom{\epsfig{file=./pic/e8_.eps,width=\linewidth}}}     & $2^{14}\cdot 3^5\cdot 5^2\cdot 7$  & 7 &{\small $|W|\cdot 4.110677 \times 10^{-5}$} & $2160$ &  \\
\hline
$A_7$&  \parbox[c]{0.16\linewidth}{\epsfig{file=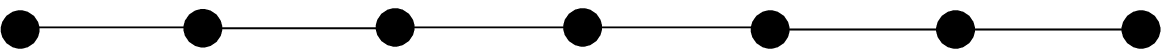,width=\linewidth}} &  \parbox[c]{0.122\linewidth}{\epsfig{file=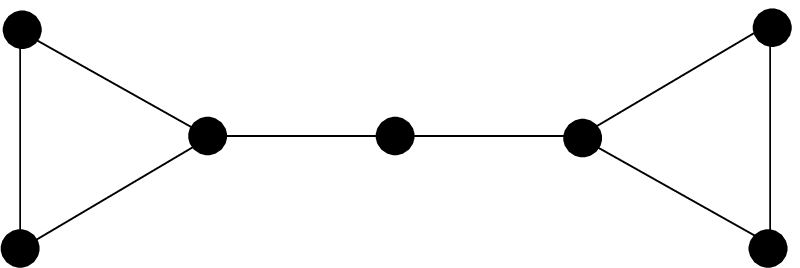,width=\linewidth}}\parbox[c]{0.17\linewidth}{\vphantom{\epsfig{file=./pic/a7_.eps,width=\linewidth}}}&$8!$&5&&70&\\
\hline
$D_8$& \parbox[c]{0.16\linewidth}{\epsfig{file=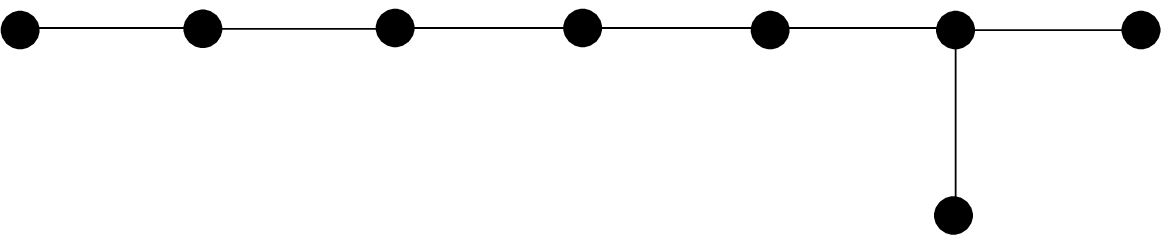,width=\linewidth}}& \parbox[c]{0.14\linewidth}{\epsfig{file=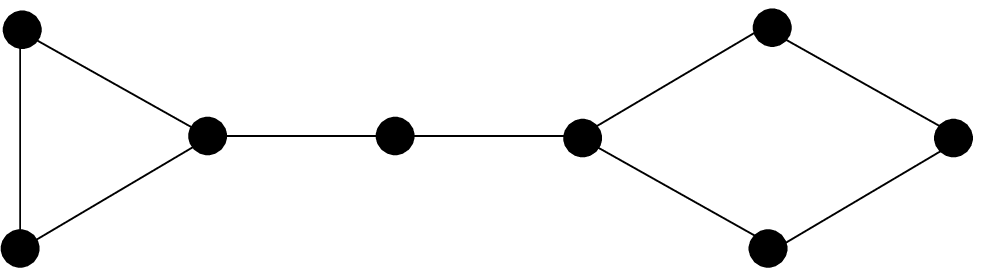,width=\linewidth}}\parbox[c]{0.2\linewidth}{\vphantom{\epsfig{file=./pic/d8_.eps,width=\linewidth}}}&$2^7\cdot 8!$&6& {\small $ |W|\cdot 0.002665$}&1120&-832\\
\hline
\end{tabular}

\end{minipage}
\end{turn}

\hss}
\vss}
\end{table}


\begin{remark}
The manifold $X$ obtained in the case $W=A_4$ is the manifold $M_5$ studied in~\cite{MPR}, it has minimal known volume amongst $5$-cusped hyperbolic $3$-manifolds. It can be obtained via a ramified $2$-fold covering over the $3$-sphere $\S^3$ branching along $1$-skeleton of a $4$-simplex. This manifold can be also obtained as the complement  of the five-chain link in the $3$-sphere.

\end{remark}



\begin{remark}
The symmetry group of $X$ always contains $W$. Further, if the underlying graph of the quiver $Q_1$ has symmetries, then the fundamental polytope of the action of $W(Q_1)$ on $X$ also has these symmetries, and thus the symmetry group of $X$ also contains a semidirect product of $W$ and the group of symmetries of the graph. In particular, one can note from Table~\ref{hyp} that the underlying graph of $Q_1$ always has an additional symmetry of order $2$, therefore we actually get the group $W\rtimes \Z_2$ acting on $X$. Moreover, in the cases $W=D_4, A_7$ or $D_8$ the symmetry  group of the fundamental polytope is $\Z_2\times \Z_2$,
which results in the action of the group $W\rtimes (\Z_2\times \Z_2)$ on the manifold $X$.

\end{remark}

\begin{remark}\label{var}
\begin{itemize}\smallskip

\item[1.] In all cases 
the fundamental domain of the action $W(Q_1)$ on $X$ is a non-compact hyperbolic polytope, i.e. the manifold $X$
has  cusps. \smallskip

\item[2.] The combinatorial type of the fundamental domain in all but the last two cases is a simplex. In the last two cases it is a pyramid over a product of two simplices.\smallskip

\item[3.] We do not know how to compute the volume of the $5$-dimensional hyperbolic pyramid which is used to construct the manifold whose symmetry group contains $A_7\rtimes (\Z_2\times\Z_2)$ (the second last line of Table~\ref{hyp}). At the same time, one can note that this pyramid consists of four copies of a pyramid with Coxeter diagram shown on Figure~\ref{IN} (the pyramid in the Table~\ref{hyp} can be obtained from one on Fig.~\ref{IN} by two consequent reflections in the facets corresponding to the leaves of the diagram). This smaller pyramid belongs to {\it Napier cycles}, a class of polytopes classified by Im Hof~\cite{ImH}. In particular, it is a simply truncated orthoscheme, and thus there is a hope to compute its volume using the technique developed by Kellerhals~\cite{Kell}.  

\begin{figure}[!h]
\epsfig{file=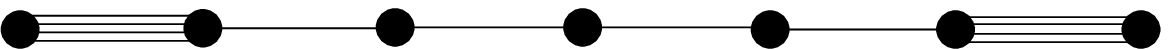,width=0.45\linewidth}
\caption{A $5$-dimensional hyperbolic Coxeter pyramid tesselating the manifold with symmetry group containing $A_7\rtimes (\Z_2\times\Z_2)$ (see Table~\ref{hyp} and Remark~\ref{var})}
\label{IN}
\end{figure}



\end{itemize}

\end{remark}

\begin{remark}
The list in Table~\ref{hyp} is complete in the following sense. Any two nodes of any quiver of finite type are connected by at most one arrow, therefore all dihedral angles of the corresponding fundamental polytope of the group action are either $\pi/2$ or $\pi/3$. All hyperbolic Coxeter polytopes satisfying such property are classified by Prokhorov in~\cite{P}. One can see that the only Coxeter diagrams whose orientations are quivers of finite type are those listed in Table~\ref{hyp}. 

On the other hand, our approach allows to construct more hyperbolic manifolds. In~\cite{P1,P2}, Parsons introduced a notion of a {\it companion basis} of a finite root system. In our terms, all companion bases (where every vector is considered up to plus/minus sign) can be described as follows. Take a quiver $Q_1$ of finite type, construct the corresponding group $W_0(Q_1)$, and construct the isomorphism between $W_0(Q_1)$ and the Weyl group $W=W(Q_1)$ by iterated changes of generators (see Section~\ref{pres}). Denote by $g_1,\dots,g_n$ the images of generating reflections of $W_0(Q_1)$ in $W(Q_1)$. Now, take the root system of $W$, and an $n$-tuple of roots corresponding to reflections $g_1,\dots,g_n$. This $n$-tuple forms a companion basis. Thus, all the manifolds constructed above originate from companion bases. Below we present an example of a hyperbolic manifold obtained from a generating set which does not correspond to a companion basis. 

As is mentioned in~\cite{BM}, not every generating set of reflections of $W$ corresponds to a companion basis. The easiest example is the quadruple of generators $(t_1,t_2,t_3,t_4)$ of $W=A_4$ given by roots $e_1-e_5,e_2-e_5,e_3-e_5,e_4-e_5$ respectively, where the root system $A_4$ consists of vectors $\{e_i-e_j\}_{1\le i,j\le 5}$ in $\R^5$ with standard basis $\{e_i\}_{1\le i\le 5}$. The reason is the following: every two of these four roots form an angle $2\pi/3$, so if there were a corresponding quiver, then its underlying graph should have been the complete graph on four vertices. However, it is easy to see that any such quiver is mutation-infinite, so it cannot be mutation-equivalent to any orientation of Dynkin diagram $A_4$.

Now we may note that the Coxeter group $$W_0=\langle u_1,u_2,u_3,u_4\mid u_i^2=(u_iu_j)^3=e\rangle$$ with Coxeter diagram shawn on Fig.~\ref{a4complete} is generated by reflections in the facets of an ideal hyperbolic $3$-dimensional simplex. Further, since $\{t_1,t_2,t_3,t_4\}$ generate $W=A_4$ and satisfy the same relations as $\{u_i\}$ above, we can write $W$ as a quotient of $W_0$ by a normal subgroup $W_C$ which is the normal closure of additional relations. It is easy to see that we can take $W_C$ as the normal closure of elements $(u_1u_2u_3u_2)^2, (u_1u_3u_4u_3)^2$ and $(u_2u_3u_4u_3)^2$. Now, one can check that $W_C$ is torsion-free, so the quotient of $\H^3$ by the action of  $W_C$ will be a finite volume hyperbolic manifold (it has $20$ cusps and symmetry group containing $A_5\rtimes A_4$, where $A_4$ is the symmetry group of the Coxeter diagram).  
    
\begin{figure}[!h]
\psfrag{1}{\small 1}
\psfrag{2}{\small 2}
\psfrag{3}{\small 3}
\psfrag{4}{\small 4}
\epsfig{file=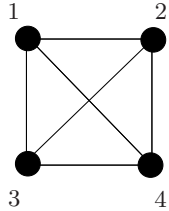,width=0.14\linewidth}
\caption{A $3$-dimensional ideal hyperbolic simplex}
\label{a4complete}
\end{figure}

\end{remark}

\section{Actions on quotients of Davis complex}
\label{Davis}

In this section we first recall the construction of Davis complex  $\Sigma\left( W,S \right)$ for each Coxeter system  $(W,S)$.
  
Then, given a Dynkin quiver $Q$ and a quiver $Q_1$ from the same mutation class, we consider the Davis complex $\Sigma(W_0(Q_1))$ for the group $W_0(Q_1)$ together with the action of the Coxeter group $W=W(Q)$ on the corresponding quotient space  $X= \Sigma(W_0(Q_1))/W_C(Q_1)$
and prove the Manifold Property for $X$ (which states that $W_C(Q_1)$ is torsion-free).
In particular, we prove that all quotients listed in Table~\ref{hyp} are hyperbolic manifolds.

\subsection{Davis complex}
For any Coxeter system $\left( W,S \right)$ there exists a contractible piecewise Euclidean cell complex $\Sigma\left( W,S\right) $ (called {\it Davis complex}) on which $W$ acts discretely, properly and cocompactly. 
The construction (based on results of Vinberg~\cite{V72}) was introduced by Davis~\cite{D1}. In~\cite{M} Moussong proved that this complex yields a natural complete piecewise Euclidean metric which is $CAT\left(0\right)$. We first give a brief description of this complex following~\cite{NV}, and then follow~\cite{sub} to define convex polytopes in $\Sigma\left( W,S\right) $.

\subsubsection{Construction of $\Sigma\left( W,S\right) $.}
For a finite group $W$ the complex $\Sigma\left(W,S\right)$ is just one cell, which is obtained
as a convex hull $K$ of the $W$-orbit of a suitable point $p$ in the standard linear 
representation of $W$ as a group generated by reflections. The point $p$ is chosen in such 
a way that its stabilizer in $W$ is trivial and all the edges of $K$ are of length 1.
The faces of $K$ are naturally identified with Davis complexes of the subgroups of $W$
conjugate to standard parabolic subgroups.

If $W$ is infinite, the complex $\Sigma\left(W,S\right)$ is built up of the Davis complexes of maximal 
finite subgroups of $W$ glued together along their faces corresponding to common 
finite subgroups. The 1-skeleton of $\Sigma\left(W,S\right)$ considered as a combinatorial graph 
is isomorphic to the Cayley graph of $W$ with respect to the generating set $S$.

The action of $W$ on $\Sigma\left(W,S\right)$ is generated by reflections.
The {\it walls} in $\Sigma\left(W,S\right)$ are the fixed points sets of reflections in $W$.
The intersection of a wall $\alpha$ with cells of $\Sigma\left(W,S\right)$ supplies $\alpha$
with a structure of a piecewise Euclidean cell complex with finitely many isometry types of cells.
Walls are totally geodesic: 
any geodesic joining two points of $\alpha$ lies entirely  in $\alpha$. 
Since $\Sigma$ is contractible and $CAT\left(0\right)$, any two points of  $\Sigma$ can be joined by a unique geodesic. 

Any wall divides $\Sigma\left(W,S\right)$ into two connected components. All the walls decompose 
$\Sigma\left(W,S\right)$ into connected components which are compact sets called {\it chambers}. Any chamber is a 
fundamental domain of the $W$-action on $\Sigma\left(W,S\right)$.  
The set of all chambers with appropriate adjacency relation is isomorphic 
to the Cayley graph of $W$ with respect to $S$.



In what follows, if $W$ and $S$ are fixed, we write  $\Sigma(W)$ or $\Sigma$ instead of $\Sigma\left(W,S\right)$.

\begin{remark}
If the group $W$ acts cocompactly by reflections on spherical, Euclidean or hyperbolic space ${\mathbb X}$,
then the Davis complex $\Sigma$ is quasi-isometric to ${\mathbb X}$.  

\end{remark}

\subsubsection{Convex polytopes in $\Sigma$}

For any wall $\alpha$ of $\Sigma$ we denote by $\alpha^+$ and $\alpha^-$
the closures of the connected components of $\Sigma\setminus \alpha$, we call these components {\it halfspaces}.

A {\it convex polytope} $P\subset \Sigma$ is an intersection of finitely many halfspaces
$P=\bigcap\limits_{i=1}^n  \alpha_i^+$, such that $P$ is not contained in any wall.
Clearly, any convex polytope $P\subset \Sigma$ is a union of closed chambers.

In what follows by writing $P=\bigcap\limits_{i=1}^n  \alpha_i^+$ we assume that
the collection of walls $\alpha_i$ is {\it minimal}: for any $j=1,\dots ,n$ we have
$P\ne \bigcap\limits_{i\ne j} \alpha_i^+$.
A {\it facet} of $P$ is an intersection $P\cap \alpha_i$ for some $i\le n$.
For any $I\subset \{1,\dots,n\}$ a set $f=\bigcap\limits_{i\in I}  \alpha_i\cap P$ is called a {\it face} of 
$P$ if it is not empty.


We can easily define a dihedral angle formed by two walls:
if $\alpha_i\cap \alpha_j\ne \emptyset$ there exists a maximal cell $C$  
of $\Sigma$ intersecting $\alpha_i\cap \alpha_j$.
We define the angle $\angle \left(\alpha_i,\alpha_j\right)$ to be equal to the corresponding Euclidean angle
formed by $\alpha_i\cap C$ and $\alpha_j\cap C$.
Clearly, any dihedral angle formed by two intersecting walls in $\Sigma$ is equal to
 $k\pi / m$ for some positive integers $k$ and $m$.  
A convex polytope $P$ is called {\it acute-angled} if each of the dihedral angles of $P$ 
does not exceed $\pi /2$. 

A convex polytope $P$ is called a {\it Coxeter polytope} if all its dihedral angles are
integer submultiples of $\pi$. In particular, a fundamental domain of any reflection subgroup of $W$
is a Coxeter polytope in $\Sigma$.  Conversely, any Coxeter polytope in $\Sigma$ is a fundamental chamber for the subgroup of $W$ generated by reflections in its walls (see e.g.~\cite[Theorem 4.9.2]{D}).


\subsection{Manifold Property}
\label{man-s}
Let us briefly recall the procedure  described in Section~\ref{various actions}:

\begin{itemize}\smallskip
\item[1.] Start with a Dynkin quiver $Q$ with corresponding Weyl group $W$. Let $Q_1$ be any quiver mutation-equivalent to $Q$.\smallskip
\item[2.] Consider the Coxeter group $W_0(Q_1)$ defined by the relations (R1) and (R2) (its Coxeter diagram coincides with
the underlying graph of the quiver $Q_1$). \smallskip
\item[3.] Let $\Sigma(W_0(Q_1))$ be the Davis complex of the group  $W_0(Q_1)$. By construction, $W_0(Q_1)$ acts on $\Sigma(W_0(Q_1))$ by reflections. \smallskip
\item[4.] Let $C_1,\dots C_p$ be the collection of all defining cycle relations for $Q_1$, define $W_C(Q_1)$ to be the normal closure of all these relations in $W_0(Q_1)$.\smallskip
\item[5.] Define $W(Q_1)= W_0(Q_1)/W_C(Q_1) $.  As is shown in~\cite{BM}, $W(Q_1)$ is isomorphic to $W$. \smallskip
\item[6.] Since $W_C(Q_1)\subset W_0(Q_1)$, $W_C(Q_1)$ acts on $\Sigma(W_0(Q_1))$.  Consider $X(Q_1)=\Sigma(W_0(Q_1))/W_C(Q_1)$. Then $W=W(Q_1)$ acts properly on $X(Q_1)$. \smallskip
\end{itemize}

So, for each quiver in the mutation class we have assigned an action of $W$ on the space $X(Q_1)$ (the latter heavily depends on the quiver $Q_1$). 


\begin{theorem}[(Manifold property)]
\label{manifold}
The group $W_C(Q_1)$ is torsion-free.

\end{theorem}

\begin{cor}
\label{man}
If  $\Sigma(W_0(Q_1))$ is a manifold then $X(Q_1)$ is also a manifold. 

\end{cor} 

Let $F$ be a fundamental domain for the action of $W(Q_1)$ on $\Sigma(W_0(Q_1))$.  
A {\it star } $\Star(f)$  of a face $f$ of $F$ is the union of all fundamental domains $F_i$  of $\Sigma$ containing $f$:
$$
\Star(f)=\bigcup\limits_{f\in F_i} F_i.
$$
We will prove the following statement which is equivalent to the Manifold Property (cf.~\cite[Corollary~2]{EM}).

\begin{theorem}
\label{manifold-eq}
Let $F$ be a fundamental domain for the action of $W(Q_1)$ on $\Sigma(W_0(Q_1))$. 
Then for any face $f$ of $F$ the interior $ \left(\Star(f)\right)^0$ of the star of $f$ in $X(Q_1)$ is isometric to the interior of the star of $f$ in $\Sigma(W_0(Q_1))$.

\end{theorem}

Equivalently, if a face $f$ of $F$ is contained in $q$ fundamental domains of the action of $W_0(Q_1)$ on $\Sigma(W_0(Q_1))$, 
then $f$ is also contained in $q$ fundamental domains of the action of $W(Q_1)$ on $X(Q_1)$.  


To prove the Manifold Property we will use the following observation.

\begin{lemma}
\label{eq-dihedral}
Let $\t t_i$ and $\t t_j$ be two generating reflections of $W_0(Q_1)$, and let $t_i$ and $t_j$ be their images in $W=W(Q_1)$ under the canonical projection $W_0(Q_1)\to W(Q_1)=W_0(Q_1)/W_C(Q_1)$. Then the orders of the elements $\t t_i\t t_j$ and $t_it_j$ coincide.
\end{lemma}

\begin{proof}
The dihedral group generated by $t_i$ and $t_j$ is a quotient of the dihedral group generated by $\t t_i$ and $\t t_j$. The order of the latter group equals $6$ or $4$, therefore, if we assume that the two groups are distinct, the order of the former equals two, i.e. $t_i=t_j$, which is impossible since a Coxeter group $W$ of rank $n$ cannot be generated by less than $n$ reflections (see e.g.~\cite[Lemma~2.1]{sub}).

\end{proof}

\begin{proof}[of Theorem~\ref{manifold-eq}]
Let $f$ be a face of $F$. Denote by $I\subset\{1,\dots,n\}$ the index set such that $f$ is the intersection of the walls fixed by $\{\t t_i\}_{i\in I}$. To prove Theorem~\ref{manifold-eq} (and thus the Manifold Property) it is sufficient to prove that the group $\wt T_I\subset W_0(Q_1)$ generated by $\{\t t_i\}_{i\in I}$ has the same order as the group  $T_I\subset W(Q_1)=W$ generated by $\{t_i\}_{i\in I}$. 

Denote $k=|I|$. The group $\wt T_I$ can be considered as a group generated by reflections in the facets of a spherical Coxeter $(k-1)$-simplex $\wt P$ with dihedral angles $\pi/m_{ij}$, where $m_{ij}$ is the order of $\t t_i\t t_j$. Furthermore, since every $t_i\in W=W(Q)$ is a reflection on the $X(Q)=\S^{n-1}$, the group $ T_I$ can also be considered as a group generated by reflections in the facets of a spherical $(k-1)$-simplex $P$. The simplex $P$ may not be Coxeter, however, by Lemma~\ref{eq-dihedral}, its dihedral angles are equal to $\pi p_{ij}/m_{ij}$ for some $p_{ij}$ being coprime with $m_{ij}$. Since $m_{ij}=2$ or $3$, this implies that either $p_{ij}=1$, or $p_{ij}/m_{ij}=2/3$. According to~\cite{F}, this means that reflections in the facets of $\wt P$ and $P$ generate isomorphic groups.    

\end{proof}

\subsection{Mutation of complexes}
\label{manifold mutation}
In this section we describe how to obtain $X(Q_2)$ from $X(Q_1)$, where $Q_2=\mu_k (Q_1)$. The construction generalizes the examples shown in Sections~\ref{torus} and~\ref{back-s}. 

\begin{itemize}\smallskip
\item[Step~1.] Let $s_1,\dots,s_n$ be the generators of $W$ corresponding to $Q_1$ (i.e., if $\{s_i\}$ with relations (R1) and (R2) are generating reflections of $W_0(Q_1)$ composing a Coxeter system, and the $\{s_i\}$ with relations (R1)--(R3) are their images under canonical homomorphism $W_0(Q_1)\to W_0(Q_1)/W_C(Q_1)$), let $t_1,\dots,t_n$ be the generators corresponding to $Q_2$. As it is shown in~\cite{BM}, $t_i$ can be obtained from $s_i$ as follows 
$$
t_{i}=
\begin{cases}
s_k s_i s_k & \text{if there is an arrow from $i$ to $k$ in $Q_1$;} \\
s_i& \text{otherwise.} \\
\end{cases}
$$

Denote by $\pi_j$, $j=1,2$, the canonical projection $\Sigma(W_0(Q_j))\to X(Q(j))$. According to the construction, a fundamental chamber $F_j\subset X(Q_j)$ of the action of $W$ on $X(Q_j)$ is an image of a fundamental chamber  $\wt F_j\subset \Sigma(W_0(Q_j))$ under the projection. 
The elements $s_i$ act on $\Sigma(W_0(Q_1))$ by reflections in the facets of $\wt F_1$, let walls $\t \alpha_1,\dots,\t \alpha_n\subset \Sigma(W_0(Q_1))$ be the mirrors of the reflections $s_1,\dots,s_n$.
Similarly, the elements $t_i$ are reflections in the facets of $\wt F_2$, let walls $\t \beta_1,\dots,\t \beta_n\subset \Sigma(W_0(Q_2))$ be the mirrors of the reflections $t_1,\dots,t_n$.\smallskip

\item[Step~2.] Consider the action of the element $t_i$ on $\Sigma(W_0(Q_1))$.
It is a reflection with respect to a wall $\t\gamma_i$ satisfying
$$
\t\gamma_i=
\begin{cases}
s_k(\t\alpha_i) & \text{ if there is an arrow from $i$ to  $k$ in $Q_1$};  \\
\t\alpha_i & \text{ otherwise.}\\
\end{cases}
$$


\item[Step~3.] The fundamental chamber $\wt F_1$ is an intersection of $n$ halfspaces bounded by walls $\{\t\alpha_i\}$. Denote these halfspaces by $\t\alpha_i^+$, i.e. $\wt F_1=\bigcap\t\alpha_i^+$. Note that the dihedral angles of a polytope $\wt F_1$ equal $\pi/m_{ij}$, where $m_{ij}$ is the order of $s_is_j$ in $W_0(Q_1)$. In particular, as all $m_{ij}=2$ or $3$, every two facets of $\wt F_1$ have a non-empty intersection.\smallskip

\item[Step~4.] Consider the domain $H\subset \Sigma(W_0(Q_1))$ defined as follows:

 $H$ is the intersection of $n$ halfspaces $H=\bigcap\limits_{i=1}^n \tilde\gamma_i^+$, where
 the walls $\tilde\gamma_i$ are defined in Step~$2$ and the half-spaces $\tilde\gamma_i^+$, $i\ne k$, are chosen by
$$
\tilde\gamma_i^+=
\begin{cases}
\tilde\alpha_k^{-} & \text{if $i=k$,} \\
s_k(\t\alpha_i^+), & \text{if there is an arrow from $i$ to $k$ in $Q$,} \\
\t\alpha_i^+ & \text{otherwise}
\end{cases}
$$
(cf. Remark~\ref{rem choice}).

Note that $H$ is well-defined, i.e. $H$ indeed has exactly $n$ facets. This follows from the fact that the facet $\t\alpha_k\cap\wt F_1$ of $\wt F_1$ (which is also a facet of $H$) has exactly $n-1$ facets $\t\alpha_k\cap\t\alpha_i\cap\wt F_1$, and by construction of $H$ all these faces of $\wt F_1$ are also faces of $H$.   \smallskip 


\item[Step~5.]
The domain $H$ is a convex (possibly, non-compact) polytope in $\Sigma(W_0(Q_1))$ with $n$ facets. 
The reflections with respect to the facets of $H$ are precisely $t_i$, they generate the group $W_0(Q_1)$,
however, in general, $H$ is not a fundamental domain of the action of $W_0(Q_1)$ on $\Sigma(W_0(Q_1))$.\smallskip
 
\item[Step~6.]
Now, we will use $H$ to construct a new cell complex $Y$ (such that the group $W$ acts on $Y$ with fundamental chamber $H$).

Consider $|W|$ copies of the polytope $H$, labeled by the elements of $W$ 
(denote by $H_w$ the polytope corresponding to $w\in W$).
 
Glue the polytope  $H_w$ to the polytope  $H_{wt_i}$ along the facets $\tilde \gamma_i$ 
of $H_w$ and  $H_{wt_i}$ so that the reflection in the common facet swaps the polytopes 
(and takes the facet  $\tilde \gamma_j$ of  $H_w$ to the facet of  $H_{wt_i}$ having the same number).   
Denote by $Y$ the space obtained by the gluing. Note that by ``reflection'' here we mean a map from one polytope to another, this may not extend to a reflection on the whole $Y$ (cf. Remark~\ref{non-reflection}).\smallskip

\item[Step~7.] Denote by $m_{ij}$ the order of the subgroup $\langle t_i,t_j \rangle\subset W$ (note that the Coxeter group $$\langle t_1,\dots t_n\,|\,(t_it_j)^{m_{ij}}=e\rangle$$ is the group $W_0(Q_2)$). Consider the following completion $Y'$ of $Y$.\smallskip
\begin{itemize}
\item First, we complete the fundamental domain $H$ (obtaining a completed fundamental chamber $H'$) in the following way.
For every pair $\{i,j\}\subset\{1,\dots n\}$ we add an $(n-2)$-dimensional face $\tilde\gamma_{i}\cap\tilde\gamma_{j}$ if it is not a face of  $H$ yet; we define the dihedral angle composed by facets $\tilde \gamma_i\cap H$ and $\tilde \gamma_j\cap H$ of $H'$ as $\pi/m_{ij}$. Then, for each finite subgroup $\langle t_i\,|\, i\in I\rangle\subset W_0(Q_2)$ (where $I\subset \{1,\dots,n \} $ is an index set) we add an $(n-|I|)$-dimensional face $\bigcap_{i\in I} \tilde\gamma_{i}$ if it is not a face of  $H$ yet. The procedure is well-defined since any subgroup of a finite group is finite. The resulting polytope $H'$ can be identified with a fundamental chamber of $\Sigma(W_0(Q_2))$, i.e. with $\wt F_2$. \smallskip

\item Now, we do this for each $H_w$, $w\in W$ (a new face $\bigcap_{i\in I} \t\gamma_{i}$  of $H_w$ is identified with the similar face of $H_{wt_j}$ if $j\in I$). The procedure results in a new complex $Y'$.\smallskip
\end{itemize}

\item[Step~8.] As $H'$ can be identified with $\wt F_2$, and $Y'$ is obtained from copies of $H'$ by gluings along the action of $W$ (see Step~6), the complex $Y'$ can be identified with $X(Q_2)=\Sigma(W_0(Q_2))/W_C(Q_2)$. \smallskip


\end{itemize}

\section{General case}
\label{sec diagr}
In all the considerations above we started with a Dynkin quiver $Q$. 
In this section 
we consider mutations of {\it diagrams} instead of quivers, which allows us to deal with non-simply-laced Dynkin diagrams. 



\subsection{Mutations of diagrams}
Recall that quivers are in one-to-one correspondence with skew-symmetric integer matrices. An integer $n\times n$  matrix $B$ is {\it skew-symmetrizable} if there exists a diagonal $n\times n$ matrix
$D=\{d_1,\dots,d_n\}$ with positive integer entries such that the product $BD$ is skew-symmetric, i.e., $b_{ij}d_j=-b_{ji}d_i$. 

Given a skew-symmetrizable matrix $B$ one obtains a {\it diagram}  $\G=\G(B)$ in the following way:
$\G$ is an oriented labeled graph with vertex set $1,\dots, n$, where vertex $i$ is connected to vertex $j$ 
by an arrow labeled by $ {|b_{ij}b_{ji}|}$ if $b_{ij}\ne 0$.

One diagram may correspond to several distinct matrices. Any diagram $\G$ constructed by a skew-symmetrizable matrix satisfies the following property: for any chordless cycle in $\G$ the product of labels along this cycle is a perfect square (cf.~\cite[Exercise~2.1]{Kac}). 

Since skew-symmetric matrices are also skew-symmetrizable, we may understand quivers as a partial case of diagrams: to make a diagram out of a quiver one needs to square all the labels. 

A {\it mutation} $\mu_k$ of a diagram $\G$ is defined similarly to the mutation of quivers~\cite{FZ2}:
  
\begin{itemize}
\item
orientations of all arrows incident to a vertex $k$ are reversed; 
\item
for every pair of vertices $(i,j)$ such that $\G$ contains arrows directed from $i$ to $k$ and from $k$ to $j$ the weight of the arrow joining $i$ and $j$ changes as described in Figure~\ref{diagrammut}.
\end{itemize}

\begin{figure}[!h]
\begin{center}
\psfrag{a}{\small $a$}
\psfrag{b}{\small $b$}
\psfrag{c}{\small $c$}
\psfrag{d}{\small $d$}
\psfrag{k}{\small $k$}
\psfrag{mu}{\small $\mu_k$}
\epsfig{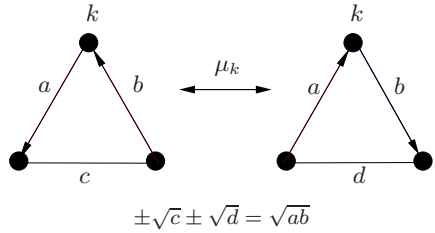}\\
\medskip
$\pm\sqrt{c}\pm\sqrt{d}=\sqrt{ab}$
\caption{Mutations of diagrams.
The sign before $\sqrt{c}$ (resp., $\sqrt{d}$) is positive if the three vertices form an oriented cycle, and negative otherwise. Either $c$ or $d$ may vanish. If $ab$ is equal to zero then neither value of $c$ nor orientation of the corresponding arrow does change.}
\label{diagrammut}

\end{center}
\end{figure}

As for quivers, a diagram is {\it of finite type} if it is mutation-equivalent to an orientation of a Dynkin diagram.

\subsection{Generalization of the construction to non-simply-laced diagrams of finite type}

The construction is generalized straightforwardly, but the definition of the group defined by a diagram is a bit longer than the one in the simply-laced case.

Let $\G$ be an orientation of a Dynkin diagram with $n$ nodes, let $W$ be the corresponding finite Weyl group, and let $\G_1$ be any diagram  mutation-equivalent to $\G$. Denote by $W(\G_1)$ the group generated by $n$ generators $s_i$ with the following relations:

\begin{itemize}
\item[(R1)] $s_i^2=e$ for all $i=1,\dots,n$;

\item[(R2)] $(s_is_j)^{m_{ij}}=e$ for all $i,j$, not joined by an edge labeled by $4$, 
where
$$
m_{ij}=
\begin{cases}
2 & \text{if $i$ and $j$ are not connected;} \\
3 & \text{if $i$ and $j$ are connected by an edge.} \\
4 & \text{if $i$ and $j$ are connected by an edge labeled by $2$;} \\
6 & \text{if $i$ and $j$ are connected by an edge labeled by $3$.} 
\end{cases}
$$

\item[(R3)] (cycle relation) for every chordless oriented cycle $\C$ given by 
$$i_0\stackrel{w_{i_0i_1}}\to i_1\stackrel{w_{i_1i_2}}\to\cdots\stackrel{w_{i_{d-2}i_{d-1}}}\to i_{d-1}\stackrel{w_{i_{d-1}i_{0}}}\to i_0$$
and for every $l=0,\dots,d-1$ we compute $t(l)=\left(\prod\limits_{j=l}^{l+d-2}\!\!\!\sqrt{w_{i_j i_{j+1}}}\ -\sqrt{w_{i_{l+d-1} i_l}}\right)^2$, where the indices are considered modulo $d$;
now for every $l$ such that $t(l)<4$, 
we take the relation

$$
(s_{i_l}\ s_{i_{l+1}}\dots s_{i_{l+d-2}}s_{i_{l+d-1}}s_{i_{l+d-2}}\dots s_{i_{l+1}})^{m(l)}=e,
$$
where 
$$
m(l)=
\begin{cases}
2 & \text{if $t(l)=0$;} \\
3 & \text{if $t(l)=1$;} \\
4 & \text{if $t(l)=2$;} \\
6 & \text{if $t(l)=3$} 
\end{cases}
$$
(this form of cycle relations was introduced by Seven in~\cite{S}).

\end{itemize}

According to~\cite{BM}, all the cycle relations for a given chordless cycle follow from one with $m(l)=2$ (such a relation always exists). Thus, as for quivers, we may take exactly one defining relation per cycle. 

It is shown in~\cite{BM} that the group $W(\G_1)$  does not depend on the choice of a diagram in the mutation class of $\G$. In particular, it is isomorphic to the initial Coxeter group $W$. 

The further part of the construction generalizes to the diagram settings verbatim. We obtain three manifolds, see Table~\ref{hyp-n}

\setlength{\tabcolsep}{0.28cm}
\begin{center}
\begin{table}[!h]
\caption{Actions on hyperbolic manifolds, non-simply-laced case.   }
\label{hyp-n}
\begin{tabular}{|c|c|c|c|c|c|c|c|}
\hline
&&&&&&&\\
$W$    & $\G$ & $\G_1$ & $|W|$ & $\dim(X)$ &
\begin{tabular}{c}$\vol{X}$ \\ {\small approx.} \end{tabular} & \begin{tabular}{c}number \\ of cusps \end{tabular} & \begin{tabular}{c}$\chi(X)$\\ ($\dim{X}$ even) \end{tabular}\\
&&&&&&&\\
\hline
{$B_3$}  &  \psfrag{2}{\scriptsize 2 }\parbox[c]{0.084\linewidth}{\epsfig{file=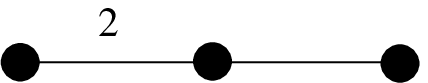,width=\linewidth}}      &   \psfrag{2}{\scriptsize 2 }\parbox[c]{0.044\linewidth}{\epsfig{file=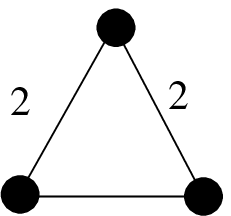,width=\linewidth}}
& $2^3\cdot 3!$& 2 & $8\pi$ & compact& -4\\
\hline
$B_4$  &  \psfrag{2}{\scriptsize 2 }\epsfig{file=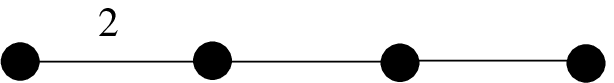,width=0.12\linewidth}      &  \psfrag{2}{\scriptsize 2 }\parbox[c]{0.08\linewidth}{\epsfig{file=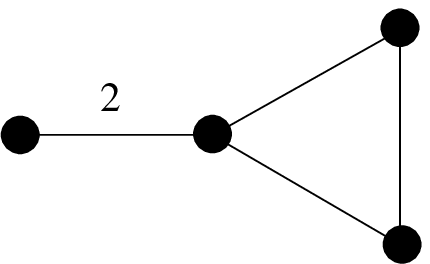,width=\linewidth}}
      & $2^4\cdot 4!$& 3 &  {\small $|W|\cdot 0.211446$  }   & 16 & \\
\hline
$F_4$  &   \psfrag{2}{\scriptsize 2 }\epsfig{file=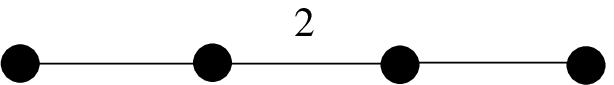,width=0.12\linewidth}     &  \psfrag{2}{\scriptsize 2 }\parbox[c]{0.06\linewidth}{\epsfig{file=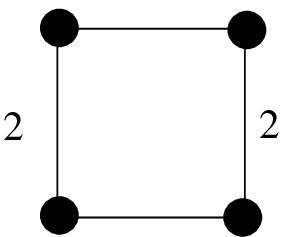,width=\linewidth}}
& $2^7\cdot 3 ^2$& 3 &  {\small $|W|\cdot 0.222228 $} & compact  &  \\
\hline
\end{tabular}
\end{table}
\end{center}

\begin{remark}
The manifold $X$ constructed for the group $B_4$ coincides with the manifold constructed in Section~\ref{act-hyp} for the group $D_4$ (see Table~\ref{hyp}). 

\end{remark}

\begin{remark}
In contrast to the simply-laced case, we do not know whether the list of hyperbolic manifolds (Table~\ref{hyp-n}) that can be obtain by our construction is complete: there is no known classification of finite volume hyperbolic Coxeter polytopes with angles $\pi/2, \pi/3, \pi/4$, so potentially the construction may provide other examples of finite volume hyperbolic manifolds.

\end{remark}

\begin{remark}
\label{lemma-n}
The proof of the Manifold Property works for diagrams, but we need to update the proof of Lemma~\ref{eq-dihedral}. 

Let us proceed by induction on the number of mutations required to obtain $\G_1$ from $\G$. Assume that the lemma holds for a diagram $\G_2$ mutation-equivalent to $\G$, and we want to deduce the lemma for $\G_1=\mu_k(\G_2)$. In the notation of Lemma~\ref{eq-dihedral} (after substituting $Q$ by $\G$) this means the following: if we denote by $\t u_i,\t u_j,\t u_k\in W_0(\G_2)$ the reflections corresponding to nodes $i,j,k$ of $\G_2$ and by $u_i,u_j,u_k\in W(\G_2)=W$ their projections to $W$, we want to deduce the equality of orders of elements $\t t_i\t t_j\in W_0(\G_1)$ and $t_it_j\in W(G_1)=W$ from the equality of the orders of $\t u_i\t u_j\in W_0(\G_2)$ and $u_iu_j\in W(\G_2)=W$.

If one of $i,j$ (say, $i$) equals $k$ the statement is obvious: in this case the label of the arrow between $i$ and $j$ is the same in $\G_1$ and $\G_2$ (and thus the orders of $\t t_i\t t_j\in W_0(\G_1)$ and $\t u_i\t u_j\in W_0(\G_2)$ are the same), and $t_it_j=(u_iu_j)^{\pm 1}$ (since either $t_j=u_j$ or $t_j=u_iu_ju_i$), so the orders of $t_it_j$ and $u_iu_j$ are also the same.

If the other case we can restrict our consideration to the subdiagrams of $\G_1$ and $\G_2$ spanned by three vertices $i,j,k$. A subdiagram of a diagram of finite type also has finite type~\cite{FZ2}, so if it is not simply-laced then $\G_1$ (and $\G_2$) is mutation-equivalent to one of the three diagrams $B_3$, $B_2+A_1$ and $G_2+A_1$. Now, a short straightforward check verifies the lemma.   

\end{remark}

\subsection{Example: action of $B_3$ on the $2$-sphere and a hyperbolic surface of genus $3$.}

Consider the diagram $\G=B_3$ (see Fig.~\ref{b3}(a)). The 
corresponding Dynkin diagram of type $B_3$  determines a finite Coxeter group $W$ acting on the $2$-dimensional sphere by reflections:
$$W=\langle s_1,s_2,s_3 \,|\, s_i^2=(s_1s_2)^3=(s_2s_3)^4=(s_1s_3)^2=e \rangle $$ 
The fundamental domain of this action is a spherical triangle with angles $(\frac{\pi}{3},\frac{\pi}{4},\frac{\pi}{2})$, see Fig.~\ref{b3}(b) for the action and a fundamental domain in stereographic projection.

\begin{figure}[!h]
\begin{center}
\psfrag{a}{\small (a)}
\psfrag{b}{\small (b)}
\psfrag{c}{\small (c)}
\psfrag{2_}{\small $2$}
\psfrag{mu}{\scriptsize $\mu_2$}
\psfrag{1}{\scriptsize $1$}
\psfrag{2}{\scriptsize $2$}
\psfrag{3}{\scriptsize $3$}
\epsfig{file=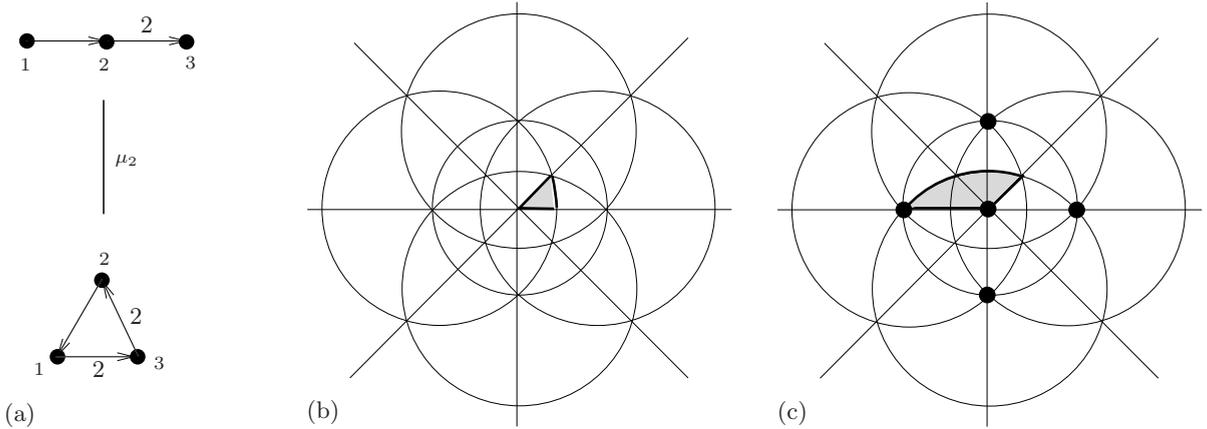,width=0.99\linewidth}
\label{b3}
\caption{(a) A diagram of type $B_3$ and its mutation $\mu_2$; (b) action of $B_3$ on the sphere, in $s$-generators (stereographic projection); (c) action of $B_3$ on the sphere in $t$-generators: bold points are the branching points for the covering (one more branching point at infinity)
}
\end{center}
\end{figure}

Now, apply the mutation $\mu_2$ to the diagram $\G$ obtaining the diagram $\G_1$ (see Fig.~\ref{b3}(a)). The corresponding group 
 $$W_0(\G_1)=\langle t_1,t_2,t_3  \ | \ t_i^2=(t_1t_2)^3=(t_2t_3)^4=(t_1t3)^4=e\rangle$$  
acts on the hyperbolic plane by reflections, and the fundamental domain of this action is a triangle $F$  with angles $(\frac{\pi}{3},\frac{\pi}{4},\frac{\pi}{4})$.
We need to take a quotient of this action by the normal closure of the element
$(t_1\ t_2t_3t_2)^2\in W_0(G_1)$.
As we expect from Manifold Property, we will get some finite volume hyperbolic surface $S$ tiled by 48 images of the 
triangle $F$. It is easy to see that $S$ is a genus $3$ surface.

Similarly to Section~\ref{sph to tor} describing $A_3$ case, one can see that $S$ can be obtained as a ramified degree $4$ covering of the sphere branching at six points composing the orbit of the ``central point'' (see Fig.~\ref{b3}(c)).

\section{Infinite groups}
\label{inf}

\subsection{Group constructed by a diagram}
The result of~\cite{FT} allows to generalize the construction to infinite groups.
More precisely, it is shown in~\cite{FT} that to a diagram $\G$ of finite mutation type that either is exceptional (see~\cite{FeSTu1,FeSTu2}) or arises from a triangulation of a unpunctured surface or orbifold (see~\cite{FST,FeSTu3}) one can assign a Coxeter group $W_0(\G)$ (using exactly the same rule as above). 

Then one takes a quotient group imposing relations of the following types:
\begin{itemize}
\item[(R3)] cycle relations;
\item[(R4)] additional affine relations;
\item[(R5)] additional handle relations;
\item[(R6)] additional $X_5$-relations. 
\end{itemize}
All these relations are of the type $(r_i r_j)^m=e$, where $r_i$ and $r_j$ are some conjugates of the generating reflections, 
see~\cite{FT} for details. 

As before, we define $W_C(\G)$ as the normal closure of all relations of types (R3)-(R6). 
It is shown in~\cite{FT}  that the group $W(\G)=W_0(\G)/W_C(\G)$ does not depend on the choice of $\G$ in a given mutation class. 
 
This implies that we can repeat verbatim the construction of actions on factors of the Davis complex $\Sigma(W_0(\G))$, so that we get an action of an (infinite) group $W=W_0(\G)/W_C(\G)$ on the quotient space $X(\G)=\Sigma(W_0(\G))/W_C(\G)$. Furthermore, as before, a mutation of $\G$ induces a mutation of $X(\G)$.  

\subsection{Diagrams of affine type}
\label{aff}
A diagram $\G$ is called a diagram of {\it affine type} if $\G$ is mutation-equivalent to an orientation of an affine Dynkin diagram different from an oriented cycle. In particular, if $\G$ is of affine type, the group $W=W(\G)$ is an affine Coxeter group.

Mutation classes of diagrams of affine type may be rather complicated (see~\cite{H} for the detailed description of mutation classes of non-exceptional diagrams of finite and affine type). In particular, a diagram $\G$ in the mutation class may contain one arrow labeled by $4$ (this corresponds to two reflections in a Coxeter system of $W_0(\G)$ generating an infinite dihedral group, or, equivalently, to two non-intersecting facets of a fundamental chamber of $\Sigma(W_0(\G)$).    

Nevertheless, applying considerations from Section~\ref{man} and Remark~\ref{lemma-n}, we obtain the Manifold Property for diagrams of affine type.  

\begin{theorem}
If $\G$ is a diagram of affine type then the group $W_C(\G)$ is torsion-free.

\end{theorem}

\begin{remark}
We are currently unable to drop the assumption that $\G$ is of affine type. 
Our proof of Manifold Property uses an induction on the number of mutations one needs to perform to obtain a diagram without oriented cycles (such diagrams are called {\it acyclic}, and diagrams from their mutation classes are {\it mutation-acyclic}). If $\G$ is not of finite or affine type, then either $\G$ is not mutation-acyclic, or it is not mutation-finite (see~\cite{FST,FeSTu1,FeSTu2}).


\end{remark}

\subsection*{Acknowledgments}
We are grateful to V.~Emery, M.~Feigin, A.~King, R.~Marsh, R.~Kellerhals and M.~Shapiro for helpful discussions. 
We also thank R.~Guglielmetti for computing the volume of one of the hyperbolic pyramids. 
The work was initiated during the program on cluster algebras at MSRI in the Fall of 2012. We would like to thank the organizers of the program for invitation, and the Institute for hospitality and excellent research atmosphere.

\affiliationone{
Anna Felikson and Pavel Tumarkin\\
Department of Mathematical Sciences, Durham University, Science Laboratories, South Road, Durham, DH1 3LE\\ 
UK
\email{anna.felikson@durham.ac.uk\\ 
pavel.tumarkin@durham.ac.uk}}
\end{document}